\documentclass{article}
\usepackage{graphicx} 

\usepackage{amsmath, appendix, ulem}
\usepackage{amsthm}
\usepackage[nobysame]{amsrefs}
\usepackage{amssymb, color}
\usepackage[margin=1.3in]{geometry}
\usepackage{mathrsfs}
\usepackage{graphicx}
\usepackage{float}
\usepackage{epsf}
\usepackage[colorlinks=true]{hyperref}
\usepackage{enumitem}
\usepackage{titletoc}
\usepackage[linesnumbered, ruled]{algorithm2e}

\usepackage{subcaption}

\usepackage{subeqnarray}
\usepackage{cases}
\graphicspath{{../Figures/}}

\numberwithin{equation}{section}

\newtheorem{lem}{Lemma}[section]
\newtheorem{thm}{Theorem}[section]
\newtheorem{proposition}[thm]{Proposition}

\newtheorem{rmk}{Remark}[section]
\newtheorem{definition}[thm]{Definition}

\newcommand{\nn}{\nonumber}
\newcommand{\R}{{\mathbb R}}

\newcommand{\barc}{\overline{c}}
\newcommand{\N}{{\mathbb N}}

\renewcommand{\tilde}{\widetilde}

\renewcommand{\bar}{\overline}

\newcommand{\bx}{\boldsymbol{x}}
\newcommand{\by}{{\boldsymbol{y}}}
\newcommand{\bX}{{\boldsymbol{X}}}

\newcommand{\bOM}{\overline{\boldsymbol{M}}}
\newcommand{\bM}{\boldsymbol{M}}

\newcommand{\mc}[1]{\mathcal{#1}}

\newcommand{\EE}{\mathbb{E}}
\newcommand{\RR}{\mathbb{R}}

\newcommand{\PP}{\mathbb{P}}

\newcommand{\OX}{\overline{X}}

\newcommand{\bxy}{\bar x_m(\bM_t^{-m})}

\usepackage[mathlines]{lineno}

\newcommand{\TE}{\mathcal{E}}

\DeclareMathOperator\diag{diag}

\title{A consensus-based algorithm for non-convex\\ multiplayer games}

\author{Enis Chenchene\thanks{Department of Mathematics and
	Scientific Computing, University of Graz, Heinrichstrasse 36, 8010 Graz, Austria, email: enis.chenchene@uni-graz.at, hui.huang@uni-graz.at} \and Hui Huang\footnotemark[1] \and Jinniao Qiu\thanks{Department of Mathematics and Statistics, University of Calgary, 2500 University Dr NW, Calgary, AB T2N 1N4, Canada, email: jinniao.qiu@ucalgary.ca}}

\date{\today}

\begin{document}

\maketitle

\begin{abstract}
In this paper, we present a novel consensus-based zeroth-order algorithm tailored for non-convex multiplayer games. The proposed method leverages a metaheuristic approach using concepts from swarm intelligence to reliably identify global Nash equilibria. We utilize a group of interacting particles, each agreeing on a specific consensus point, asymptotically converging to the corresponding optimal strategy. This paradigm permits a passage to the mean-field limit, allowing us to establish convergence guarantees under appropriate assumptions regarding initialization and objective functions. Finally, we conduct a series of numerical experiments to unveil the dependency of the proposed method on its parameters and apply it to solve a nonlinear Cournot oligopoly game involving multiple goods.
\end{abstract}
{\small {\bf Keywords:} Non-convex, multiplayer games, Nash equilibrium, swarm optimization, Laplace's principle.}
\section{Introduction}
 
Multiplayer games \cite{narahari2014game}, ranging from strategic board games to complex economic systems, have always fascinated researchers due to their intricate dynamics and strategic interactions among multiple players. Understanding and analyzing the outcomes of these games is crucial in various fields such as economics \cite{king2012understanding}, social sciences \cite{gokhale2014evolutionary}, and computer science \cite{fan2021fault}. In recent years, significant progress has also been made in developing advanced learning approaches within the field of artificial intelligence towards multiplayer scenarios. One notable example is the extension of adversarial learning, which was originally applied to settings with a single generator and discriminator, to accommodate multiple agents \cites{song2018multi, zhao2020improving, li2017triple}.  This process can also be observed in reinforcement learning, where multiplayer game theory has been integrated to enhance learning algorithms \cites{busoniu2008comprehensive, lanctot2017unified, dai2018sbeed}.

 In this paper, we focus on exploring a common class of non-convex games involving multiple players, specifically $M\geq 3$ players. Each player, denoted by $m \in [M] := \{1, \ldots, M\}$, aims to minimize their own cost function $\TE_m(x_m;\bx_{-m}): \mathbb{R}^{Md} \rightarrow \mathbb{R}$. The cost function is influenced by two factors: the player's own decision, represented by $x_m \in \mathbb{R}^d$, and the decisions made by all other players, denoted as $\bx_{-m}=(x_1,\dots,x_{m-1},x_m,\dots,x_M) \in \mathbb{R}^{(M-1)d}$. This setting  has garnered significant attention in machine learning applications. For instance, in the context of sensor localization \cites{ke2017distributed, yang2018df}, the decision variable $x_m$ corresponds to a sensor node, and the cost function $\TE_m$ is instantiated as the Euclidean norm. Similarly, in the domain of robust neural network training \cites{nouiehed2019solving, deng2021local}, $x_m$ represents the model parameter, and $\TE_m$ is the cross-entropy function. Furthermore, the considered setting has the potential to inspire solutions for resource allocation problems in unmanned vehicles \cite{yang2019energy} and secure transmission \cite{ruby2015centralized}. In these contexts, the decision variable $x_m$ represents the allocation of transmit resources, and the cost function $\TE_m$ captures the associated transmission cost. 
 
 With the formulation presented above, it is both natural and crucial, from the perspectives of both game theory and machine learning, to pursue the identification of the global Nash Equilibrium (NE) \cite{nash1950equilibrium}, which represents a widely recognized concept of optimality in game theory, wherein no player can improve their outcome by unilaterally altering their strategy while keeping the strategies of others unchanged. We specialize this concept in the following definition:
\begin{definition}\label{nash}
	Point $\bx^*:=(x_1^*,\dots, x_M^*)$ is a NE of cost functions $\{\TE_1,\dots, \TE_M\}$ if
	\begin{equation}
		x_m^* \in \arg\min_{x_m\in \RR^d}\TE_m(x_m;\bx^*_{-m})\,, \quad \text{for all} \ m \in [M]\,.
	\end{equation}
\end{definition}
\noindent
 
The search for global NE in non-convex settings remains an open problem \cites{maciel2003global,liu2021variance}. This challenge arises not only due to the absence of powerful tools compared to those available in convex scenarios but also due to the diversity of non-convex structures, which may require specific methodologies for resolution. While efficient techniques have been developed within convex conditions that have led to significant advancements in multiplayer game models \cites{yi2019operator, chen2021distributed}, these approaches may prove insufficient when confronted with non-convexity. In such cases, they may become trapped in local NE or approximate solutions while following pseudo-gradients, failing to converge to a global NE. Moreover, while inspiring breakthroughs have been made in solving non-convex two-player min-max games in different contexts, such as Polyak--Łojasiewicz cases \cites{nouiehed2019solving, fiez2021global} or concave cases \cites{lin2020gradient, rafique2022weakly}, these advancements may not readily apply to multiplayer settings. This is because the global stationary conditions in multiplayer games are intricately coupled and cannot be independently addressed by each player. Consequently, novel effective methods are required for attaining them, which is exactly the main contribution of this article

In this paper we introduce a novel zero-order consensus-based approach for finding global NE points in multiplayer games. This method draws inspiration from the consensus-based optimization (CBO) framework initially introduced in \cites{carrillo2018analytical,PTTM}. CBO belongs to the family of global optimization methodologies, which leverages systems of interacting particles to achieve consensus around global minimizers of the cost functions. As part of the broader class of metaheuristics \cites{Blum:2003:MCO:937503.937505,Gendreau:2010:HM:1941310}, CBO orchestrates interactions between local improvement procedures and global strategies, utilizing both deterministic and stochastic processes. This interplay ultimately results in an efficient and robust procedure for exploring the solution space of the cost functions.
More importantly, the CBO approach possesses an inherent advantage of being gradient-free. This characteristic makes it particularly desirable when dealing with cost functions that lack smoothness or when computing their derivatives is computationally expensive. 

 Motivated by diverse applications, researchers have extended and adapted the original CBO model to encompass various settings. These extensions include incorporating memory effects or gradient information \cites{riedl2022leveraging,totzeck2020consensus,cipriani2022zero}, integrating momentum \cite{chen2022consensus}, and employing jump-diffusion processes \cite{kalise2023consensus}. Additionally, CBO has been extended to address global optimization on compact manifolds \cites{fornasier2020consensus,ha2022stochastic}, handle general constraints \cites{borghi2023constrained,carrillo2023consensus},  cost functions with multiple minimizers \cite{bungert2022polarized}, multi-objective problems \cites{borghi2022consensus,borghi2022adaptive}, and sampling from distributions \cite{carrillo2022consensus}. Furthermore, CBO has also been applied to tackle high-dimensional machine learning problems \cites{carrillo2021consensus,fornasier2021consensus}, saddle point problems \cite{huang2022consensus}, asset allocation problems \cite{bae2022constrained}, and more recently the clustered federated learning \cite{carrillo2023fedcbo}. It has been demonstrated that CBO exhibits behavior similar to stochastic gradient descent (SGD) \cite{riedl2023gradient}.
 Specifically, our CBO dynamic for $M$-player games is a collection of $M\cdot N$ particles $\{(\OX_t^{1,i},\dots, \OX_t^{M,i})\}_{i\in [N]}\in \R^{Md}$ satisfying the following Stochastic Differential Equations (SDEs)
\begin{equation}\label{particle}
		d\OX_t^{m,i}=-\lambda(\OX_t^{m,i}-X_\alpha(\bar{\rho}_t^{m}))dt+\sigma D(\OX_t^{m,i}-X_\alpha(\bar{\rho}_t^{m}))dB_t^{m,i}\,,\quad i\in[N]\,, \ m\in[M]\,,
\end{equation}
where $\lambda,\sigma>0$ are drift and diffusion parameters, $\{B_t^{m,i}\}_{m\in [M],i\in [N]}\in \R^d$ are standard independent Brownian motions, and $D(X):=\diag(X_1,\dots,X_d)$ (anisotropic), or $D(X):= | X|$ (isotropic), for all $X\in \R^d$. However, in the subsequent analysis, we will specifically focus on the anisotropic case.
The system is complemented with identically and independently distributed (i.i.d.) initial data $\{(\OX_0^{1,i},\dots, \OX_0^{M,i})\}_{i\in [N]}$ with respect to the common law $\bar{\rho}_0$.
The first term on the right hand-side of \eqref{particle} is a deterministic drift that pulls the particles towards a current consensus point, computed as a convex combination of particles locations as
\begin{equation}\label{XaN}
X_\alpha(\bar{\rho}_t^{m}):=\frac{1}{\|\omega_{\alpha}^{\mc{E}_m}(\cdot;\bOM_t^{-m})\|_{L^1(\bar{\rho}_t^{m})}} \int_{\RR^{d}}x_m\omega_{\alpha}^{\mc{E}_m}(x_m;\bOM_t^{-m})\bar{\rho}_t^{m}(dx_m)\,,
\end{equation}
where $\bar{\rho}_t^{m}(dx_m):=\frac{1}{N}\sum_{i=1}^N\delta_{\OX_t^{m,i}}dx_m$ is the empirical measure of the particle system at time $t\geq 0$, $\omega_{\alpha}^{\mc{E}_m}$ is a weight function defined as
\begin{equation}\label{eq:weight_function}
\omega_{\alpha}^{\mc{E}_m}(x;\boldsymbol{y}):=\exp\left(-\alpha \TE_m(x;\boldsymbol{y})\right)\,, \quad \text{for all} \ x \in \R^d\,, \ \boldsymbol{y}\in \R^{(M-1)d}\,,
\end{equation}
and $\bOM_t$ is a random vector given by
\begin{equation}
	\bOM_t:=\bigg(\frac{\OX_t^{1,1}+\dots +\OX_t^{1,N}}{N}, \dots, \frac{\OX_t^{M,1}+\dots +\OX_t^{M,N}}{N}\bigg)\,.
\end{equation}
The introduction of the second stochastic term in \eqref{particle} is intended to promote exploration of the energy landscape of the cost function. When the consensus is achieved, meaning $\OX_t^{m,i}=X_\alpha(\bar{\rho}_t^{m})$, both the drift and diffusion terms vanish. The choice of the weight function \eqref{eq:weight_function} comes from  the  well-known Laplace's principle \cites{miller2006applied,Dembo2010}, which states that for any probability measure $\mu\in\mc{P}( \RR^d )$, there holds for any fixed $\by\in\R^{(M-1)d}$,
\begin{equation}\label{lap_princ}
	\lim\limits_{\alpha\to\infty}\left(-\frac{1}{\alpha}\log\left(\int_{\RR^{(M-1)d} }\omega_\alpha^{\TE_m}(x_m;\by)\mu(d x_m)\right)\right)=\inf\limits_{x_m \in \rm{supp }(\mu)} \TE_m(x_m;\by)\,.
\end{equation}

The theoretical convergence analysis of a CBO method can be conducted using two different approaches. The first approach involves investigating the microscopic particle level directly, as demonstrated in \cites{ha2021convergence, ha2020convergence}, by analyzing the particle system \eqref{particle}. The second approach, which the present paper adopts, is to consider the macroscopic level and focus on its corresponding mean-field equation in the large particle limit. This mean-field approach has been successfully utilized in previous works \cites{fornasier2022anisotropic, fornasier2022convergence, fornasier2021consensus1, huang2023global}.
Indeed, as the number of particles $N\to\infty$, the mean-field limit \cites{bolley2011stochastic,huang2020mean,carrillo2019propagation,sznitman1991topics,jabin2017mean} result would imply that our CBO particle dynamic \eqref{particle} well approximates solutions of the following  mean-field  kinetic Mckean--Vlasov type equations
\begin{equation}\label{MVeq}
dX_t^m=-\lambda(X_t^m-X_\alpha(\rho_t^m))dt+\sigma D(X_t^m-X_\alpha(\rho_t^m))dB_t^m \,,\quad m\in[M]\,,
\end{equation}
where the consensus point is defined similarly to \eqref{XaN} as 
\begin{equation}\label{Xa}
X_\alpha({\rho}_t^{m}):=\frac{1}{\|\omega_{\alpha}^{\mc{E}_m}(\cdot;\bM_t^{-m})\|_{L^1(\rho_t^{m})}} \int_{\RR^{d}}x_m\omega_{\alpha}^{\mc{E}_m}(x_m;\bM_t^{-m}){\rho}_t^{m}(dx_m)\,,
\end{equation}
where $\rho_t^m$ is the distribution of $X_t^m$, and $\bM_t$ is the vector given by
\begin{equation}
	\bM_t:=\left(\EE[X_t^1], \dots, \EE[X_t^M]\right).
\end{equation}

A direct application of It\^{o}'s formula yields that the law $\rho_t$ of $\boldsymbol{X}_t:=(X_t^1, \dots, X_t^M)$ for $t\geq 0$ is a weak solution to the following  nonlinear Vlasov--Fokker--Plank equation
\begin{align}\label{meanPDE}
\partial_{t} \rho_t&=\sum_{m=1}^{M}\bigg(\lambda\nabla_{x_m} \cdot\left(\left(x_m-{X}_{\alpha}(\rho_t^m)\right) \rho_t\right)+\frac{\sigma^{2}}{2} \sum_{k=1}^d\partial^2_{(x_m)_k(x_m)_k}\left((x_m-{X}_{\alpha}(\rho_t^m))_k^{2} \rho_t\right)\bigg)\,,
\end{align}
with the initial data $\rho_0$ being the law of $\boldsymbol{X}_0$. The current paper centers on investigating the convergence  of the proposed CBO variant in the context of finding global NE in multiplayer games. While the rigorous mean-field approximation is an interesting avenue for further exploration, it falls outside the scope of the current work and is best pursued in future research endeavors. The well-posedness results regarding the CBO particle system \eqref{particle} and its mean-field dynamics \eqref{MVeq} have also been omitted from this work. However, it is worth noting that these results closely resemble the well-posedness theorems presented in \cite[Theorem 3, Theorem 6]{huang2022consensus}, where a CBO dynamic is introduced for two player zero-sum games.  Moreover, in the following we shall assume the solution to \eqref{MVeq} has the regularity $\sup_{t\in[0,T]}\EE[|X_t^m|^4]<\infty$ for all $m\in [M]$ and any time horizon $T>0$, which can be guaranteed by the assumption on the initial data $\EE[|X_0^m|^4]<\infty$ for all $m\in [M]$.

The main objective of this work is to establish the convergence of the dynamics $\boldsymbol{X}_t$ in \eqref{MVeq} to the global NE point $(x_1^*,\dots,x_M^*)$ as $t$ approaches infinity. Inspired by \cite{fornasier2021consensus1}, we introduce the variance functions as follows:
\begin{equation}\label{eq:defi_V_m}
	 V^m(t) =\EE[|X_t^m-x_m^*|^2]\,, \quad \text{for each} \ m\in[M]\,, \quad \text{and} \quad V(t) := \sum_{m=1}^M V^m(t)\,, \quad \text{for} \ t>0\,.
\end{equation}
By analyzing the decay behavior of the (cumulative) variance function $V(t)$ we  establish the convergence of the CBO dynamics to the global NE point. Specifically, we demonstrate that the function $V(t)$ decays exponentially, with a decay rate controllable through the parameters of the CBO method. This also implies the convergence of the mean-field PDE \eqref{meanPDE} to a Dirac delta centering at the global NE with respect to the 2-Wasserstein distance, i.e.
\begin{equation*}
	W_2(\rho_t,\delta_{(x_1^*,\dots,x_M^*)})\to 0\,,\quad \mbox{as } \alpha, \ t\to \infty\,.
\end{equation*}

The rest of the paper is organized as follows. In Section \ref{sec2}, we commence by outlining the assumptions governing the cost functions $\{\TE_m\}_{m\in [M]}$ (Assumptions \ref{ass:1}--\ref{ass:4}) and deriving an estimate for the variance function \eqref{eq:defi_V_m} (Lemma \ref{lemV}). Subsequently, we establish a quantitative estimate of the Laplace principle (Proposition \ref{propX}) together with a result demonstrating that the probability mass around the target NE point remains bounded away from zero at any time $0<t<\infty$ (Proposition \ref{propositive}). By leveraging these outcomes, we finally demonstrate that the variance $V(t)$ exhibits exponential decay, achieving any desired level of accuracy $\varepsilon>0$ (Theorem \ref{thm:global_convergence}). In Section \ref{sec3}, we conduct a series of numerical experiments to evaluate the performance of our proposed CBO method and analyze its dependency on parameters. Additionally, we apply the method to solve a nonlinear non-convex Cournot's oligopoly game involving multiple goods.

\section{Global convergence}\label{sec2}

In this section, we present our main result about the global convergence in mean-field law for cost functions satisfying the following conditions.

	\paragraph{Assumptions.} In this paper,  the cost functions $\{\TE_m\}_{m\in [M]}\subset \mc{C}(\R^{Md})$ satisfy:
	\begin{enumerate}[label=(A\arabic*)]
		\item \label{ass:1} There exists a unique global NE point $(x_1^*,\dots,x_M^*)$ and $|\TE_m|\leq \bar c$ for all $m\in M$.
		\item \label{ass:2} For any $m\in [M]$ and $\by\in \RR^{(M-1)d}$, there exists a unique $\bar x_m(\by)$ such that
		\begin{equation}
			\bar x_m(\by)=\arg\min_{x_m\in\R^{d}}\TE_m(x_m;\by)\,.
		\end{equation}
		Especially, it holds that $\bar x_m(\bx_{-m}^*)=x_m^*$. Moreover, there exist $\bar{c}_1,\barc_2>0$ independent of $m$ such that
		\begin{align}
		&|\bar x_m(\by_1)-\bar x_m(\by_2)|\leq \bar{c}_1|\by_1-\by_2|\,,\quad  \mbox{for any } m\in[M]\mbox{ and }\ \by_1,\by_2\in \RR^{(M-1)d} \label{eq:lipschitz_minizing} \\
		&	|\bar x_m(\by)|\leq \bar{c}_2 \,,\quad  \mbox{for any } m\in[M]\mbox{ and } \ \by\in \RR^{(M-1)d}\label{boundedbar}
		\,.
	\end{align}
		\item \label{ass:3}  For each  $m\in [M]$ and $\by\in \RR^{(M-1)d}$, there exist $\TE_\infty,\eta,\nu, R_0>0$ independent of $m$  such that
		\begin{equation}
			\eta|x_m-\bar x_m(\by)|\leq |\TE_m(x_m;\by)-\TE_m(\bar x_m(\by);\by)|^\nu\,, \quad \text{for all} \ x_m\in B_{R_0}(\bar x_m(\by))\,,
		\end{equation}
		\begin{equation}
			\TE_\infty<\TE_m(x_m;\by)-\TE_m(\bar x_m(\by);\by)\,, \quad \text{for all} \ x_m\in B_{R_0}(\bar x_m(\by))^c \,.
		\end{equation} 
     \item \label{ass:4}  For each $q>0$ there exists some $r\in (0,R_0]$ such that 
		\begin{equation}
			\sup_{m\in[M]} \sup_{\by\in\RR^{(M-1)d}}\sup_{\ x_m\in B_{r}(\bar x_m(\by)) }|\TE_m(x_m;\by)-\TE_m(\bar x_m(\by);\by)|\leq q \,.
		\end{equation}
	\end{enumerate}
	Note that $B_r(x)$, here and in the sequel, represents the ball centered in $x\in \R^d$ with radius $r>0$ with respect to the $\ell^{\infty}$ norm, i.e., $B_r(x):=\{x' \in \R^d : \max_{k \in [d]}|(x')^k-x^k| \leq r \}$.
	\begin{rmk}
	The assumptions \ref{ass:1}, \ref{ass:3}, and \ref{ass:4} above are standard in the analysis of the CBO methods. In fact, these also appear, e.g., in \cite[Definition 9]{fornasier2021consensus1}. Assumption \ref{ass:2} is tailored to our multiplayer setting, accounting for the fact that each player's strategy is contingent upon the decisions made by the other players. Specifically, later in the analysis, we show that Assumption \ref{ass:2}, coupled with a quantitative Laplace principle, is indeed crucial for obtaining a bound for $|x_m^*-X_\alpha(\rho_t^m)|$, with $m \in [M]$. It is worth noting that our numerical experiments suggest that the assumption in \eqref{eq:lipschitz_minizing} could potentially be relaxed or even dropped (cf.~Section \ref{sec:results_and_discussions}), which we leave for future work. Moreover, the assumptions do not require any convexity/concavity or differentiability conditions. For instance, consider the nonconvex, nonconcave, and nondifferentiable functions $\mathcal E_m(x_m;\bx_{-m})= \Big|\tanh\Big(x_m-  (4+|\tilde \bx_m|)^{-1} \sin\Big(\frac{\tilde \bx_{-m}}{2}\Big)\Big)\Big|\cdot (2+\sin(x_m))^{-1}$ with $\tilde \bx_{-m}=M^{-1} \sum_{j\neq m}x_j$ for $m\in [M],\, d=1,\, (x_m,\bx_{-m})\in\mathbb R\times \mathbb R^{M-1}$, which lead to a unique global  NE point with $x_1^*=x_2^*=\dots = x_M^{*}=0$ and have all the assumption \ref{ass:1}--\ref{ass:4} satisfied.
	\end{rmk}

We now present a differential inequality for the variance function $ V^m(t)$ according to \eqref{eq:defi_V_m} , whose proof is postponed to the Appendix.
\begin{lem}\label{lemV}
	Let $\{\TE_m\}_{m \in [M]}$ satisfy Assumptions \ref{ass:1}--\ref{ass:4}, and we assume that $\lambda, \sigma$ satisfy
$2\lambda-\sigma^2>0$. Then for any $m\in[M]$ the variance function $ V^m(t)$ satisfies
\begin{align}
	\frac{d  V^m(t)}{dt}\leq -(2\lambda-\sigma^2) V^m(t)+2(\lambda+\sigma^2)V^m(t)^{\frac{1}{2}}|x_m^*-X_\alpha(\rho_t^m)|+\sigma^2 |x_m^*-X_\alpha(\rho_t^m)|^2\,.
\end{align}
\end{lem}

The main idea underpinning our main convergence result consists of showing that
\begin{equation}
	\frac{d  V^m(t)}{dt}\leq -\frac{1}{2}(2\lambda-\sigma^2) V^m(t)\,,
\end{equation}
which formally requires that
\begin{equation}
	2(\lambda+\sigma^2)V^m(t)^{\frac{1}{2}}|x_m^*-X_\alpha(\rho_t^m)|+\sigma^2 |x_m^*-X_\alpha(\rho_t^m)|^2\leq \frac{1}{2}(2\lambda-\sigma^2) V^m(t)\,.
\end{equation}
This condition is fulfilled by ensuring that $|x_m^*-X_\alpha(\rho_t^m)|$ remains sufficiently small. To establish this, we shall utilize a quantitative estimate of the Laplace principle in conjunction with \eqref{eq:lipschitz_minizing} in Assumption \ref{ass:2}. Indeed, for each $m \in [M]$, the former provides a bound for $|\bar{x}_m(\bM_t^{-m})-X_{\alpha}(\rho_t^m)|$, while the latter allows us to turn it into a bound for $|x_m^*-X_{\alpha}(\rho_t^m)|$. This reveals a distinction between our multiplayer game setting and the consensus-based methods for optimization (see, e.g., \cite{fornasier2021consensus1}).

\subsection{Quantitative Laplace principle}

The following proposition yields a quantitative estimate of the Laplace principle \eqref{lap_princ} by using the inverse continuity Assumption \ref{ass:3}. To do so, for each player $m \in [M]$, we need to quantify the maximum discrepancy for the cost function $\TE_m$ around the corresponding best strategy with respect to $\bM_t^{-m}$, namely, for $r >0$,  
\begin{equation}\label{eq:definition_of_E_r}
	\TE_r^m(\bM_t^{-m}):=\sup_{x\in B_r(\bxy)}\left|\TE_m(x;\bM_t^{-m})-\TE_m(\bxy;\bM_t^{-m})\right|\,.
\end{equation}
We have the following result:
\begin{proposition}\label{propX}
	Assume that $\{\TE_m\}_{m \in [M]}$ satisfy Assumptions \ref{ass:1}--\ref{ass:4}. For any  $m\in[M]$, $t>0$ and $r\in (0,R_0]$, let $\TE_r^m(\bM_t^{-m})$ be defined as in \eqref{eq:definition_of_E_r}, $0<q\leq \frac{\TE_\infty}{2}$, and $r:=\max\{s\in(0,R_0]:~\sup_{m\in[M], \boldsymbol{y}\in \R^{(M-1)d} }\TE_s^m(\boldsymbol{y})\leq q\}$. Then, it holds that
	\begin{align}\label{eqquan}
		|\bxy-X_\alpha(\rho_t^m)|\leq& \frac{(2q)^\nu}{\eta}+\frac{\exp\left(-\alpha q\right)}{ \rho_t^m (B_r(\bxy))} \int|x-\bxy|\rho_t^m(dx)\,.
	\end{align}
\end{proposition}
\begin{rmk}
	Such $r>0$ exists because of Assumption \ref{ass:4}, and it is independent of $m$ and $\bM_t^{-m}$. The right hand-side of equation \eqref{eqquan} can become small by selecting a sufficiently small value for $q$ and a suitably large value for $\alpha$.
\end{rmk}
\begin{proof}
  	To simplify notation, let us denote by $\bar{x}_t^m:=\bar{x}_m(\bM_t^{-m})$. Let $\tilde r\geq r>0$. Using Jensen's inequality one can deduce
   \begin{equation}\label{eq:from_jensen}
   |\bar{x}_t^m-X_\alpha(\rho_t^m)|\leq \int_{B_{\tilde r}(\bar{x}_t^m)}|x-\bar{x}_t^m|\frac{\omega_\alpha^{\TE_m}(x;\bM_t^{-m})\rho_t^m(dx)}{\|\omega_\alpha^{\TE_m}(\cdot;\bM_t^{-m})\|_{L_1(\rho_t^m)}}\nn+\int_{B_{\tilde r}(\bar{x}_t^m)^c}|x-\bar{x}_t^m|\frac{\omega_\alpha^{\TE_m}(x;\bM_t^{-m})\rho_t^m(dx)}{\|\omega_\alpha^{\TE_m}(\cdot;\bM_t^{-m})\|_{L_1(\rho_t^m)}}\,.
   \end{equation}
The first term is bounded by $\tilde r$ since $|x-\bar{x}_t^m|\leq \tilde r$ for all $x\in B_{\tilde r}(\bar{x}_t^m)$. Let us now tackle the second term. It follows from Markov's inequality that
\begin{align}\label{eq:denominator_1}
\|\omega_\alpha^{\TE_m}&(\cdot,\bM_t^{-m})\|_{L_1(\rho_t^m)}\nn\\
\geq& \exp\Big(-\alpha(\TE_r^m(\bM_t^{-m})+\TE_m(\bar{x}_t^m;\bM_t^{-m}))\Big)\cdot\nn\\
&\quad \rho_t^m\left(\bigg\{x:\exp\left(-\alpha\TE_m(x;\bM_t^{-m})\right)\geq \exp\Big(-\alpha(\TE_r^m(\bM_t^{-m})+\TE_m(\bar{x}_t^m;\bM_t^{-m}))\Big)\bigg\}\right)\nn\\
=&	\exp\Big(-\alpha(\TE_r^m(\bM_t^{-m})+\TE_m(\bar{x}_t^m;\bM_t^{-m}))\Big)\cdot\nn\\
&\quad \rho_t^m \left(\bigg\{x: \TE_m(x;\bM_t^{-m})\leq \Big(\TE_r^m(\bM_t^{-m})+\TE_m(\bar{x}_t^m;\bM_t^{-m})\Big)\bigg\}\right)\nn\\
\geq& \exp\left(-\alpha\Big(\TE_r^m(\bM_t^{-m})+\TE_m(\bar{x}_t^m,\bM_t^{-m})\Big)\right)\rho_t^m\left(B_r(\bar{x}_t^m)\right)\,.
\end{align}
Hence, using that
\begin{equation*}
\omega_\alpha^{\TE_m}(x;\bM_t^{-m}) \leq C_{\tilde{r}}:= \exp\Big(-\alpha\inf_{x\in B_{\tilde r}(\bar{x}_t^m)^c}\TE_m(x;\bM_t^{-m})\Big)\,, \quad \text{for all} \ x \in \R^d,
\end{equation*}
and denoting by $C_r$ the right hand-side of \eqref{eq:denominator_1}, for the second term we have
\begin{align}
\int_{B_{\tilde r}(\bar{x}_t^m)^c}&|x-\bar{x}_t^m|\frac{\omega_\alpha^{\TE_m}(x;\bM_t^{-m})\rho_t^m(dx)}{\|\omega_\alpha^{\TE_m}(\cdot;\bM_t^{-m})\|_{L_1(\rho_t^m)}}\nn\\
&\leq \frac{1}{C_r}\int_{B_{\tilde r}(\bar{x}_t^m)^c}|x-\bar{x}_t^m|\omega_\alpha^{\TE_m}(x;\bM_t^{-m})\rho_t^m(dx)\leq \frac{C_{\tilde{r}}}{C_r}\int_{B_{\tilde r}(\bar{x}_t^m)^c}|x-\bar{x}_t^m|\rho_t^m(dx)\,.
\end{align}
Thus from \eqref{eq:from_jensen}, for any $\tilde r\geq r>0$ we obtain
\begin{equation}\label{14}
	|\bar{x}_t^m-X_\alpha(\rho_t^m)|\leq \tilde{r}+\frac{C_{\tilde{r}}}{C_r} \int_{B_{r}(\bar{x}_t^m)}|x-\bar{x}_t^m|\rho_t^m(dx)\,,
\end{equation}
where, recalling the definition of $C_{\tilde{r}}$ and $C_r$ above, 
\begin{equation}\label{eq:C_over_C}
	\frac{C_{\tilde{r}}}{C_r} =\frac{1}{\rho_t^m\left(B_{r}(\bar{x}_t^m)\right)} \exp\left(-\alpha\Big(\inf_{x\in B_{\tilde r}(\bar{x}_t^m)^c}\TE_m(x;\bM_t^{-m})-\TE_r^m(\bM_t^{-m})-\TE_m(\bar{x}_t^m;\bM_t^{-m})\Big)\right)\,.
\end{equation}
Next, we choose $\tilde r=\eta^{-1}(q+\TE_r^m(\bM_t^{-m}))^\nu$, and note that by Assumption \ref{ass:2} this is indeed a suitable choice since
\begin{equation*}
	\begin{aligned}
		\tilde r&\geq \eta^{-1}\TE_r^m(\bM_t^{-m})^\nu\\
		&=\eta^{-1}\bigg(\sup_{x\in B_r(\bar{x}_t^m)}|\TE_m(x;\bM_t^{-m})-\TE_m(\bar{x}_t^m;\bM_t^{-m})|\bigg)^\nu\geq \sup_{x\in B_r(\bar{x}_t^m)}|x-\bar{x}_t^m|=r \,.
	\end{aligned}
\end{equation*}
It only remains to see that with this choice, \eqref{14} and \eqref{eq:C_over_C} turn into \eqref{eqquan}. Notice that
\begin{equation*}
	\inf_{x\in B_{\tilde r}(\bar{x}_t^m)^c}\TE_m(x;\bM_t^{-m})-\TE_m(\bar{x}_t^m;\bM_t^{-m}) \geq
 \begin{cases}
     (\tilde r\eta)^{\frac{1}{\nu}}\,, & \ \mbox{if} \ x\in B_{\tilde r}(\bar{x}_t^m)^c \cap B_{R_0}(\bar{x}_t^m)\,,\\
      \TE_\infty\,, & \ \mbox{if} \ x\in B_{R_0}(\bar{x}_t^m)^c \,.
 \end{cases}
\end{equation*}
Thus, since $(\tilde r\eta)^{\frac{1}{\nu}}=q+\TE_r^m(\bM_t^{-m})\leq 2q\leq \TE_\infty$, we have
\begin{equation*}
    \inf_{x\in B_{\tilde r}(\bar{x}_t^m)^c}\TE_m(x;\bM_t^{-m})-\TE_m(\bar{x}_t^m;\bM_t^{-m}) \geq q+\TE_r^m(\bM_t^{-m})\,.
\end{equation*}
Therefore
\begin{equation*}
	\inf_{x\in B_{\tilde r}(\bar{x}_t^m)^c}\TE_m(x,\bM_t^{-m})-\TE_r^m(\bM_t^{-m})-\TE_m(\bar{x}_t^m,\bM_t^{-m})
	\geq q+\TE_r^m(\bM_t^{-m})-\TE_r^m(\bM_t^{-m})=q\, .
\end{equation*}
In particular, from \eqref{eq:C_over_C} we get  $C_{\tilde{r}}/{C_r}\geq \exp(-\alpha q)[\rho_t^m(B_r(\bar{x}_t^m))]^{-1}$. Inserting the latter and the definition of $\tilde r$ into \eqref{14}, we obtain the claim.
\end{proof}

To eventually apply the above proposition, one needs to ensure that $\rho_t^m (B_r(\bxy))$ is bounded away from $0$ for a finite time horizon $T$ for every $m \in [M]$. To do so, we employ a rather technical argument inspired from \cite[Proposition 23]{fornasier2021consensus1}, and introduce the mollifier $\phi_r^{\bar\bx}\colon \R^{Md}\to\R$ defined by
\begin{equation}\label{lemcutoff}
	\phi_{r}^{\bar\bx}(\bx):=
	\begin{cases}\prod_{m=1}^{M}\prod_{k=1}^{d} \exp \left(1-\frac{r^{2}}{r^{2}-\left(x_m-\bar x_m \right)_{k}^{2}}\right)\,, & \ \text{if} \ \bx=(x_1,\dots, x_m) \in \boldsymbol{B}_r(\bar \bx)\,,\\
	0\,, & \ \text{else}\,,
	\end{cases}
\end{equation}
where $\boldsymbol{B}_r(\bar \bx):=B_r(\bar x_1)\times \dots \times B_r(\bar x_m)$, for $r > 0$, and $\bar \bx:=(\bar x_1, \dots, \bar x_M)$ is any point in $[-\barc_2,\barc_2]^M$. Then, we get the following result, whose proof is postponed to the Appendix.
\begin{proposition}\label{propositive}
	Let $\phi_r$ be the mollifier defined in \eqref{lemcutoff} and $(\bX_t)_{0\leq t\leq T}$ be the solution to \eqref{MVeq} up to a time $T>0$. Assume that $\sup_{m\in[M],t\in[0,T]}\left|x_m^{*}-X_{\alpha}\left(\rho_{t}^m\right)\right|\leq B$ for some $B>0$. Then, there exists some constant $\vartheta>0$ depending only on $d,\lambda,\sigma,r,M,\bx^*,\barc_2$ and $B$ such that
\begin{equation}\label{eq:propositive}
	\inf_{\bar \bx\in [-\barc_2,\barc_2]^M}\PP(\bX_t\in \boldsymbol{B}_r(\bar \bx))\geq \inf_{\bar \bx\in [-\barc_2,\barc_2]^M} \EE[\phi_r^{\bar\bx}(\bX_t)]\geq \inf_{\bar \bx\in [-\barc_2,\barc_2]^M}\EE[\phi_r^{\bar\bx}(\bX_0)]\exp(-\vartheta  t)
\end{equation}
holds, where $\barc_2$ comes from \eqref{boundedbar} and $\bx^*:=(x_1^*, \dots, \bar x_M^*)$ is the unique NE point.
\end{proposition}

\subsection{Global convergence in mean-field law}\label{sec:global_convergence}
Now we are ready to prove the global convergence result, which provides a rate of decreasing for the variance function $V(t)$ defined in \eqref{eq:defi_V_m}, for $t$ within a prescribed time-range. 

\begin{thm}\label{thm:global_convergence} Let $(\bX_t)_{t\geq 0}$ be a solution to \eqref{MVeq}. Assume that $\{\TE_m\}_{m \in [M]}$ satisfy Assumptions \ref{ass:1}--\ref{ass:4}, and $\lambda, \sigma$ satisfy
\begin{equation}\label{eq:conditions_lambda_and_sigma}
	2\lambda > \sigma^2\,, \quad \text{and} \quad \bar{c}_1\leq \frac{1}{4}\min\left\{\frac{2\lambda-\sigma^2}{8\sqrt{M}(\lambda+\sigma^2)},\sqrt{\frac{2\lambda-\sigma^2}{4M\sigma^2}}\right\}\,.
\end{equation}	
Let $\varepsilon >0$ be any accuracy, and assume that the initial data satisfies $V(0)\geq 2\varepsilon$, 
\begin{equation}\label{initial-preparation}
\rho_0^m(B_{r_0}(\bar x_m(\textbf{M}_0^{-m})))>0\text{ for all } m\in [M],\quad  \text{and}\quad
 \inf_{\bar \bx\in [-\barc_2,\barc_2]^M}\EE[\phi_{r_1}^{\bar\bx}(\bX_0)]>0,
 \end{equation}
  where $r_0,r_1>0$ are to be determined later. Then for $\alpha$ sufficiently large there exists some $0<T_\ast\leq T_\varepsilon$ such that
\begin{equation}\label{eq:rate_for_V}
	 V(t)\leq V(0)\exp\left(-\frac{2\lambda-\sigma^2}{2}t\right)\,,\quad \mbox{for} \ t\in[0,T_\ast)\,,
\end{equation}
and $V(T_\ast)\leq \varepsilon$. Here $T_\varepsilon:=\frac{2}{2\lambda-\sigma^2}\log(\frac{V(0)}{\varepsilon})$. 
\end{thm}
\begin{rmk}
The theorem above asserts that for any given accuracy $\varepsilon>0$, we can choose $\alpha$ to be sufficiently large (dependent on $\varepsilon$) such that the variance $V(t)$ (dependent on $\alpha$) of the CBO dynamic decreases exponentially until it achieves the desired accuracy of $\varepsilon$. Conversely, for any fixed $\alpha$, the achievable accuracy $\varepsilon$ is constrained and depends on the value of $\alpha$. This can be seen in \eqref{eq:dependency_of_alpha_with_epsilon}. In addition, the initial preparation \eqref{initial-preparation} may be satisfied for all $r_0,r_1>0$ if the density function of $\bX_0$ is continuous and has support in the whole space $\mathbb R^{Md}$; for instance, this may be easily satisfied when generating initial samples of  $\bX_0$ with Gaussian distributions.
\end{rmk}
	\begin{proof}
	First, we define 
		\begin{equation}
			T_\alpha:=\inf\bigg\{t\geq0\,:\,  V(t)\leq  \varepsilon \quad \mbox{or} \quad \sum_{m= 1}^M|x_m^*-X_\alpha(\rho_t^m)|\geq 2Mc_3V(0)^\frac{1}{2} \bigg\}\,,
		\end{equation}
	where $c_3$ is the constant given by
	\begin{equation}\label{eq:defi_c3}
		c_3:=\min\left\{\frac{2\lambda-\sigma^2}{8\sqrt{M}(\lambda+\sigma^2)},\sqrt{\frac{2\lambda-\sigma^2}{4M\sigma^2}}\right\}\,.
	\end{equation}
	As in Proposition \ref{propX}, we simplify notation denoting by $\bar{x}_t^m:=\bxy$ for all $t \in [0, T_\alpha)$. Firstly, it follows from Proposition \ref{propX} that for
	\begin{equation*}
		q_0:=\frac{1}{2}\min\bigg\{\bigg(\frac{\eta c_3 V(0)^{\frac{1}{2}}}{2}\bigg)^{1/\nu},\TE_\infty\bigg\}\,, \quad \text{and} \quad r_0:=\max\left\{s\in[0,R_0]: \sup_{m\in[M]}\sup_{\by\in\RR^{(M-1)d}}\TE_s^m(\by) \leq q_0\right\}\,,
	\end{equation*}
	which, by construction, satisfy $r_0\leq R_0$ and $\TE_{r_0}^m(\boldsymbol{M}^{-m}_0)\leq \frac{\TE_\infty}{2}$ with $\TE_{r_0}^m(\boldsymbol{M}^{-m}_0)$ defined as in \eqref{eq:definition_of_E_r}, we have the following estimate
	\begin{equation}\label{eq:estimate_laplace_variance}
		|\bar{x}^m_0-X_\alpha(\rho_0^m)|\leq  \frac{(2q_0)^\nu}{\eta}+\frac{\exp\left(-\alpha q_0\right)}{ \rho_0^m\left(B_{r_0}(\bar{x}^m_0)\right)}\sqrt{2}(1+\bar{c}_1) V(0)^{\frac{1}{2}}\,,
	\end{equation}
	where we have used that
	\begin{equation*}
		\int|x-\bar{x}^m_0|\rho_0^m(dx) \leq V^m(0)^{\frac{1}{2}}+ \bar{c}_1 \bigg(\sum_{k\neq m}^{M} V^k(0)\bigg)^{\frac{1}{2}} \leq \sqrt{2}(1+\bar{c}_1) V(0)^{\frac{1}{2}}\,,
	\end{equation*}
which holds since, by Assumption \ref{ass:2},
\begin{align*}
	|x-\bar{x}^m_0|&\leq|x-x_m^*|+ |x_m^*-\bar{x}^m_0|\\
	&=|x-x_m^*|+ |\bar x_m(\bx_{-m}^*)-\bar x_m(\boldsymbol{M}^{-m}_0)|\leq |x-x_m^*|+  \bar{c}_1 \bigg(\sum_{k\neq m}^{M} V^k(0)\bigg)^{\frac{1}{2}}.
\end{align*}
Then, using $\bar{c}_1\leq \frac{c_3}{4}$, Assumption \ref{ass:2} and the definition of $q_0$, it holds that
\begin{align*}
|x_m^*-X_\alpha(\rho_0^m)|&\leq |x_m^*-\bar{x}^m_0|+ |\bar{x}^m_0-X_\alpha(\rho_0^m)|\\
&\leq \frac{c_3}{4}V(0)^\frac{1}{2}+\frac{c_3}{2}V(0)^{\frac{1}{2}}+\frac{\exp\left(-\alpha q_0\right)}{ \rho_0^m\left(B_{r_0}(\bar{x}^m_0)\right)}\sqrt{2}(1+\bar{c}_1)V(0)^{\frac{1}{2}}\\
&\leq \bigg(\frac{3}{4}c_3+\frac{\exp\left(-\alpha q_0\right)}{\rho_0^m\left( B_{r_0}(\bar{x}^m_0)\right)}\sqrt{2}(1+\bar{c}_1) \bigg)V(0)^{\frac{1}{2}}\,.
\end{align*}
If we now choose $\alpha$ sufficiently large, e.g., $\alpha\geq \sup_{m\in[M]}\alpha_0^m$ with
\begin{equation*}
	\alpha_0^m=\frac{-\log(c_3 \rho_0^m (B_{r_0}(\bar{x}^m_0)))+\log(4\sqrt{2}(1+\bar{c}_1))}{q_0}\,,
\end{equation*}
one has
\begin{equation*}
    \sum_{m= 1}^M|x_m^*-X_\alpha(\rho_0^m)|\leq Mc_3 V(0)^{\frac{1}{2}}\,.
\end{equation*}
This together with the assumption that $V(0)\geq 2\varepsilon$ implies $T_\alpha>0$ . 

According to the definition of $T_\alpha$ one has 
\begin{equation*}
	V(t)>\varepsilon\,, \quad \text{and} \quad \sum_{m=1}^M|x_m^*-X_\alpha(\rho_t^m)|< 2Mc_3 V(0)^{\frac{1}{2}}\,, \quad \text{for all} \ t \in [0, T_\alpha)\,,
\end{equation*}
and at $t=T_\alpha$, it holds that $V(T_\alpha)\leq \varepsilon$ or  $\sum_{m=1}^M|x_m^*-X_\alpha(\rho_{T_\alpha}^m)|\geq2Mc_3 V(0)^{\frac{1}{2}}$. Next, we prove that $T_\alpha\leq T_\varepsilon$ and $V(t)$ decreases exponentially on $[0,T_\alpha)$.

\textbf{Case} $T_\alpha\leq T_\varepsilon$: Arguing as in \eqref{eq:estimate_laplace_variance}, we get for every $t>0$
\begin{align}
	|\bar{x}_t^m-X_\alpha(\rho_t^m)|
\leq \frac{(2q_1)^\nu}{\eta}+\frac{\exp\left(-\alpha q_1\right)}{ \rho_t^m (B_{r_1}(\bar{x}_t^m))}\sqrt{2}(1+\bar{c}_1) V(t)^{\frac{1}{2}}
	\,,
\end{align}
where $q_1$ and $r_1$ are constants defined by
\begin{equation}
	q_1:=\frac{1}{2}\min\bigg\{\bigg(\frac{\eta c_3\sqrt{\varepsilon}}{2}\bigg)^{1/\nu}\,,\TE_\infty\bigg\}, \quad \text{and} \quad r_1:=\max\left\{s\in[0,R_0]: \sup_{m\in[M]}\sup_{\by\in\RR^{(M-1)d}}\TE_s^m(\by) \leq q_1\right\}\,.
\end{equation}
Then, we obtain
\begin{align}
	|x_m^*-X_\alpha(\rho_t^m)|&\leq |x_m^*-\bar{x}_t^m|+ |\bar{x}_t^m-X_\alpha(\rho_t^m)|\nn\\
	&\leq \frac{c_3}{4}V(t)^\frac{1}{2}+\frac{c_3}{2}V(t)^{\frac{1}{2}}+\frac{\exp\left(-\alpha q_1\right)}{\rho_t^m \left(B_{r_1}(\bar{x}_t^m)\right)}\sqrt{2}(1+\bar{c}_1)V(t)^{\frac{1}{2}}\nn\\
	&\leq \frac{3c_3}{4}V(t)^\frac{1}{2}+\frac{\exp\left(-\alpha q_1\right)}{\inf_{|\bar x_m|\leq \barc_2}\rho_t^m \left(B_{r_1}(\bar{x}_m)\right)}\sqrt{2}(1+\bar{c}_1)V(t)^{\frac{1}{2}}
	\,.
\end{align}
Moreover, \eqref{eq:propositive} implies that 
\begin{align*}
&\inf_{|\bar x_m|\leq \barc_2}\rho_t^m \left(B_{r_1}(\bar{x}_m)\right)=\inf_{|\bar x_m|\leq \barc_2} \PP\big(X_t^m\in B_{r_1}(\bar{x}_m)\big)\nn\\
\geq&\inf_{\bar \bx\in [-\barc_2,\barc_2]^M}\PP(\bX_t\in \boldsymbol{B}_{r_1}(\bar \bx))\geq \inf_{\bar \bx\in [-\barc_2,\barc_2]^M} \EE[\phi_{r_1}^{\bar\bx}(\bX_t)]\geq \inf_{\bar \bx\in [-\barc_2,\barc_2]^M}\EE[\phi_{r_1}^{\bar\bx}(\bX_0)]\exp(-\vartheta  t)
\end{align*}
  where $\vartheta$ depends only on $d,\lambda,\sigma,M,\barc_2,\bx^*$ and $B=c_3 V(0)^{\frac{1}{2}}$.  
Since  $c_3 V(0)^{\frac{1}{2}}$ is actually independent of $\alpha$, we know $\vartheta $ is independent of $\alpha$.
This allows us to conclude that for all $t\in[0,T_\alpha)$
\begin{align}\label{eq59}
		&|x_m^*-X_\alpha(\rho_t^m)|
		\leq \frac{3c_3}{4}V(t)^{\frac{1}{2}}+\frac{2\exp\left(-\alpha q_1\right)\exp(\vartheta  T_\varepsilon)}{\inf_{\bar \bx\in [-\barc_2,\barc_2]^M}\EE[\phi_{r_1}^{\bar\bx}(\bX_0)]}\sqrt{2}(1+\bar{c}_1) V(t)^{\frac{1}{2}}\leq c_3 V(t)^{\frac{1}{2}}
		\,,
\end{align}
where we choose $\alpha\geq \alpha_1$ with
\begin{equation}\label{eq:dependency_of_alpha_with_epsilon}
	\alpha_1:=\frac{\vartheta  T_\varepsilon-\log(c_3\inf_{\bar \bx\in [-\barc_2,\barc_2]^M}\EE[\phi_{r_1}^{\bar\bx}(\bX_0)])+\log(8\sqrt{2}(1+\bar{c}_1))}{q_1}\,.
\end{equation}
From the upper bound for the time derivative of $V^m(t)$ given in Lemma \ref{lemV}, using \eqref{eq59} and the definition of $c_3$ in \eqref{eq:defi_c3}, we deduce that
\begin{equation}
		\sum_{m= 1}^M\frac{d  V^m(t)}{dt}\leq-\frac{2\lambda-\sigma^2}{2}\sum_{m= 1}^M V^m(t)\,,
\end{equation}
which, using Gronwall's inequality, leads to
\begin{equation}
	V(t)\leq V(0)\exp\left(-\frac{2\lambda-\sigma^2}{2}t\right)\,, \quad \mbox{for} \ t\in[0,T_\alpha)\,.
\end{equation}
This implies
\begin{equation}
	|x_m^*-X_\alpha(\rho_{T_\alpha}^m)|\leq c_3 V(T_\alpha)^{\frac{1}{2}}=c_3V(T_\alpha)^\frac{1}{2}\leq  c_3 V(0)^{\frac{1}{2}}\,,\quad \text{for all} \ m\in [M]\,,
\end{equation}
namely, $\sum_{m= 1}^M|x_m^*-X_\alpha(\rho_{T_\alpha}^m)|\leq Mc_3 V(0)^{\frac{1}{2}}$. Then, according to the definition of $T_\alpha$ we must have $V(T_\alpha)\leq \varepsilon$.
	
\textbf{Case} $T_\varepsilon<T_\alpha$: By the definition of $T_\alpha$ we know that $V(t)\geq \varepsilon$ and $\sum_{m= 1}^M|x_m^*-X_\alpha(\rho_t^m)|\leq2Mc_3 V(0)^{\frac{1}{2}}$ for all $t\in[0,T_\varepsilon]$. Then, arguing as the in the first case, we conclude 
\begin{equation*}
		V(t)\leq V(0)\exp\left(-\frac{2\lambda-\sigma^2}{2}t\right)\quad \mbox{for} \ t\in[0,T_\varepsilon].
\end{equation*}
This estimate and the fact that $T_\varepsilon=\frac{2}{2\lambda-\sigma^2}\log(\frac{V(0)}{\varepsilon})$ imply $V(T_\varepsilon)\leq \varepsilon$, which is a contradiction. Thus this case can never happen.
\end{proof}

\section{Numerical experiments}\label{sec3}

In this section, we present our numerical experiments, which are performed in Python on a 12th-Gen.~Intel(R) Core(TM) i7--1255U, $1.70$--$4.70$ GHz laptop with $16$ Gb of RAM and are available for reproducibility at \href{https://github.com/echnen/CBO-multiplayer}{https://github.com/echnen/CBO-multiplayer}.

As usual in CBO schemes, we discretize the interacting particle system in \eqref{particle} with a Euler--Maruyama time discretization scheme \cite{euler_scheme_1}, resulting in the method depicted in Algorithm \ref{alg:discretized_scheme}.

\normalem
\begin{algorithm}[t]
	\caption{Consensus-based zeroth-order method for multiplayer games.}\label{alg:discretized_scheme}
	\KwData{$d \in \N$, $M\in \N$, an instance of a NE problem according to Definition \ref{nash} with $\mathcal{E}_m$ for $m \in [M]$}
	\KwResult{$(x^\alpha_1,\dots, x_M^\alpha)\in (\R^d)^M$ approximate solution to the NE problem}
	\textbf{Algorithm's parameters:} $N \in \N$, $\lambda, \sigma, \alpha, T >0$, $K\in \N$ and set $dt = T/K$\\
	\textbf{Initialize:} $X_m^i \in \R^d$ for $i \in [N]$ and $m \in [M]$\\
	\While{$k \leq K$}{
		Compute $M_m = \tfrac{1}{N}\sum_{i=1}X^i_m$ for all $m \in [M]$\\
		\For{$m \in [M]$}{
		Set $\boldsymbol{M}^{-m}=(M_1, \dots, M_{m-1}, M_{m+1}, \dots, M_M)$\\
		Let $E_m^i=\mathcal{E}_m(X_m^i; \boldsymbol{M}^{-m})$ for all $i \in [N]$\\
		Find $E_m^*=\min_{i \in [N]} \ E_m^i$\\
		Compute $W_m^i = \exp\{ -\alpha (E_m^i-E_m^*) \}$ for all $i \in [N]$\\
		}
		Compute $x^\alpha_m = \left[\sum_{i=1}^N (W_i^m X_i^m)\right]/\left[\sum_{i=1}^N W_i^m\right]$ for $m \in [M]$  \hfill \# compute consensus\\
		Sample $ d B^i_m\sim \mathcal{N}(0, \sqrt{dt})$ for all $m \in [M]$ and $i \in [N]$\\
		\For{$i \in [N]$ and $m \in [M]$}{
		\eIf{Anisotropic}{
			$X^{i}_m \leftarrow X^i_m -\lambda(X^i_m-x^\alpha_m) dt + \sigma \diag (X^i_m - x^\alpha_m)  d B^i_m$ \hfill \# update distributions\\
		}{
			$X^{i}_m \leftarrow X^i_m -\lambda(X^i_m-x^\alpha_m) dt + \sigma  |  X^i_m - x^\alpha_m |   d B^i_m$\\
		}}
		$k=k+1$
	}
\end{algorithm}
\ULforem

\subsection{One-dimensional illustrative example} We start with an illustrative example with $d=1$, which provides us with an initial insight into the numerical performance of Algorithm \ref{alg:discretized_scheme}. Consider the NE problem with
\begin{equation}\label{eq:numerics_1d_non_perturbed}
	\widetilde{\mathcal{E}}_m(x_m; \bx_{-m}) := \frac{1}{2}\bigg(a_m x_m - \sum_{i\neq m}^Mx_i - b_m\bigg)^2\,, \quad \text{for all} \ m \in [M]\,,
\end{equation} 
where $a_m>M, b_m \in \R$ for each $m \in [M]$, and let $\bx^*:=(x_1^*, \dots, x_M^*)$ be the unique equilibrium point. To test the performance of Algorithm \ref{alg:discretized_scheme}, we introduce in \eqref{eq:numerics_1d_non_perturbed} a non-convex perturbation as follows  
\begin{equation}\label{eq:numerics_1d_perturbed}
	{\mathcal{E}}_m(x_m; \bx_{-m}) = \widetilde{\mathcal{E}}_m(x_m; \bx_{-m})  + \mathcal{R}(x_m-x_m^*)\,, \quad \text{for all} \ m \in [M]\,,
\end{equation}
where $\mathcal{R}:\R\to \R$ is given by $\mathcal{R}(t):= 10 (1 - \cos(10 t)) + t^2$ for $t \in \R$. Note that, in this way, the (unique) equilibrium point for \eqref{eq:numerics_1d_perturbed} is once again $\bx^*$.

To illustrate the typical behavior of Algorithm \ref{alg:discretized_scheme}, we present in Figure \ref{fig:time_snapshots} four different time snapshots of the particles for each player, alongside the corresponding cost functions, the current consensus point, and the NE. For more in-depth information, consult Figure \ref{fig:time_snapshots}.

\begin{figure}[t]
	\includegraphics[width=\linewidth]{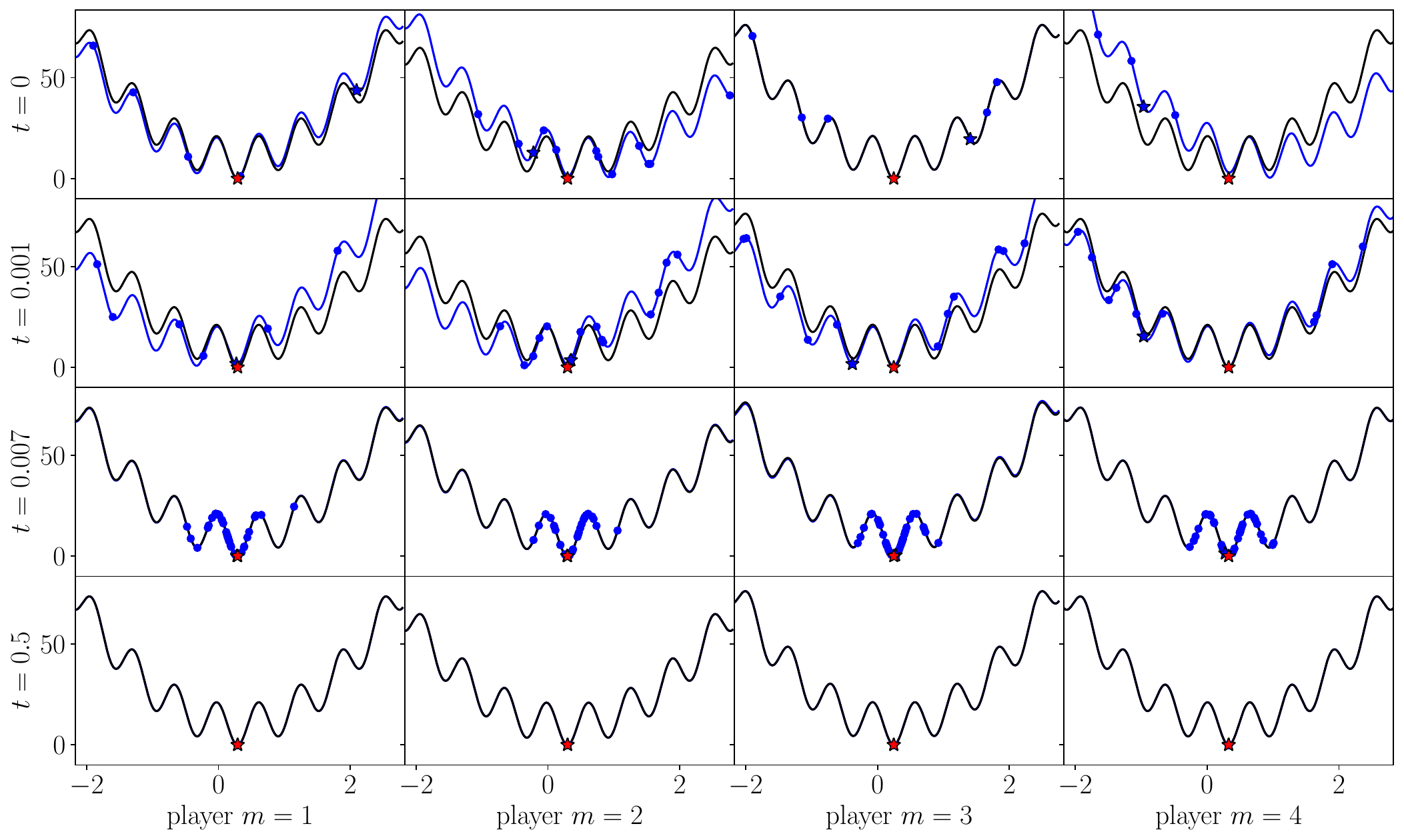}
	\caption{Different time-snapshots of the particles corresponding to the four players at different times. The blue points represent the particles. In each case, the blue curve represents the cost function of player $m$, with respect to the current consensus points of the other players, namely the function $x\mapsto \mathcal{E}_m(x; \bx_{-m}^\alpha)$. The black curve represents the loss function of player $m$ with respect to optimal choices of the other players, i.e., the function $x\mapsto \mathcal{E}_m(x; \bx^*_{-m})$. The blue stars represent the current consensus points, while the red stars represent the NE point.}
	\label{fig:time_snapshots}
\end{figure} 

\subsubsection{Experimental setup} We now investigate the influence of the algorithm's parameters on the convergence behavior of Algorithm \ref{alg:discretized_scheme} applied to solve \eqref{eq:numerics_1d_perturbed} with $M=4$. Specifically, we consider the following parameter settings:
\begin{enumerate}
	\item (Dependence on $\alpha$) We set $N = 40$, $dt = 10^{-4}$, $\sigma = 10^{-1}$, $\lambda = (10^4 + \sigma^2)/2$ and consider $500$ different choices of $\alpha$ spaced evenly in a logarithmic scale from $\alpha = 10^{-6}$ to $\alpha = 10^7$.\label{item:dependence_on_alpha}
	\item (Dependence on $\lambda$) We set $N = 40$, $dt = 10^{-4}$, $\sigma = 10^{-1}$, $\alpha = 10^7$, and consider $500$ different choices of $\lambda$ with $\lambda = (u + \sigma^2)/2$ for $500$ choices of $u$ spaced evenly in a logarithmic scale from $u = 10^2$ to $u = 10^4$.\label{item:dependence_on_lambda}
	\item (Dependence on $N$) We set $dt = 10^{-4}$, $\sigma = 10^{-1}$, $\lambda = (10^4 + \sigma^2)/2$ and consider $500$ different choices of $N$ spaced evenly in a linear scale from $N = 4$ to $N = 4000$.\label{item:dependence_on_N}
	\item (Joint dependence on $\lambda$ and $\sigma$)\label{item:dependence_on_lambda_and_sigma} In this experiment, we test the joint dependence on $\lambda$ and $\sigma$. To do so, we fix three different values of $N$, namely $N=10, 100$, or $N=1000$, $\alpha=10^7$, $dt = 10^{-2}$, and pick $100$ choices of $\lambda$ and $\sigma$ spaced evenly in a logarithmic scale from $10^{-1}$ to $10^{2.5}$ and from $10^{-1}$ to $10^{1.2}$, respectively. 
\end{enumerate}
We let the algorithm run for $K=100$ iterations, i.e., up to time $T=K~ dt  = 0.01$ for cases \ref{item:dependence_on_alpha}, \ref{item:dependence_on_lambda} and \ref{item:dependence_on_N}, and $T=0.1$ for case \ref{item:dependence_on_lambda_and_sigma}. We sample the starting measures from a Gaussian distribution centered in $\bx^* + (-2, 1, 0, 3)$ with variance $5$. Whenever possible, i.e., in cases \ref{item:dependence_on_alpha}, \ref{item:dependence_on_lambda} and \ref{item:dependence_on_lambda_and_sigma} we consider the same starting point for every independent run. In all the above cases, at each iteration of the method we measure the cumulative variance, i.e., 
\begin{equation}
	V(t):= \sum_{m=1}^M V^m(t) = \sum_{m=1}^M \bigg( \frac{1}{N}\sum_{i=1}^N |X_t^{m, i}-x_m^*|^2\bigg)\,,
\end{equation}
which, for $N$ large enough, should converge to zero with a rate according to Theorem \ref{thm:global_convergence}. We report in Figure \ref{fig:dependency_on_aln} the variance as a function of time for all the independent runs in cases \ref{item:dependence_on_alpha}, \ref{item:dependence_on_lambda} and \ref{item:dependence_on_N}. For case \ref{item:dependence_on_lambda_and_sigma}, we display in Figure \ref{fig:dependence_on_lambda_and_sigma} the logarithmic value of the variance at the final iteration for $N=10, 100$ and $N=1000$.

\subsubsection{Results and discussion}\label{sec:results_and_discussions}

In the first plot in Figure \ref{fig:dependency_on_aln}, which corresponds to the parameters setting in case \ref{item:dependence_on_alpha}, we see that the rate for $V(t)$ seems to be independent on $\alpha$, which only affects the final accuracy of the method. Specifically, even with a very small $\alpha$ of the order of $10^{-6}$ the variance still decreases in the initial iterations, but quickly stagnates. On the other hand, for a large $\alpha$ of the order of $10^7$, the final accuracy reaches $10^{-9}$. Note that this does not contradict Theorem \ref{thm:global_convergence}, as $T_\epsilon$ in \eqref{eq:rate_for_V} depends on $\epsilon$, which in turn depends on $\alpha$ as shown in \eqref{eq:dependency_of_alpha_with_epsilon}. In particular, for $\alpha$ small, $T_\epsilon$ might be small too and, thus, the rate in \ref{eq:rate_for_V} would be guaranteed only for a potentially short time.

In the second plot in Figure \ref{fig:dependency_on_aln}, which corresponds to the parameters setting in case \ref{item:dependence_on_lambda}, we plot the variance decrease as a function of time for different choices of $\lambda$, but with fixed $\sigma$ and $\alpha$. Here, we can see that indeed the choice of $\lambda$ seems only to affect the rate for $V(t)$, as also suggested by Theorem \ref{thm:global_convergence}, but does not affect the final accuracy. Specifically, very large choices of $\lambda$ lead to fast optimization methods, but might be unstable.

In the third plot in Figure \ref{fig:dependency_on_aln}, which corresponds to the parameters setting in case \ref{item:dependence_on_N}, we see that the rate of $V(t)$ does not seem to depend on the choice of $N$, expect for some very small values on $N$, for which the method is not able to find the NE. Remarkably, also the final accuracy, i.e., $V(t)$ for $t=T$ does not seem to be affected by the choice of $N$. 

Eventually, Figure \ref{fig:dependence_on_lambda_and_sigma}, which corresponds to experiment in case \ref{item:dependence_on_lambda_and_sigma}, confirms that the requirement $2\lambda- \sigma^2>0$ in Theorem \ref{thm:global_convergence} is in fact quite sharp. Indeed, the converging area (in blue) falls into the region for which $2\lambda>\sigma^2$. Note that the green region in the left-bottom corner of the image does not indicate either a converging, nor a diverging behavior, since $\lambda$ is too small to provide evidence of convergence in the prescribed time range, since the rate might be too small, see Figure \ref{fig:dependency_on_aln} (Middle). However, from Figure \ref{fig:dependency_on_aln} (Left and Right), we might consider the parameters corresponding to the area delineated by the black curve, i.e., those cases in which $V(T) < V(0)$, as leading to a converging (but potentially slow) methods. It is interesting to notice that the latter is perfectly aligned with the curve $2\lambda-\sigma^2=0$, especially with $N=1000$. It is also worth mentioning that the plots in Figure \ref{fig:dependency_on_aln} (Left and Right) indicate that the information displayed in Figure \ref{fig:dependence_on_lambda_and_sigma} may not depend heavily on $\alpha$. Figure \ref{fig:dependence_on_lambda_and_sigma} also reveals that if $N$ is not large, choosing $\lambda$ too large might lead to methods that do not converge to the NE with high precision. This suggests that leaving some room for exploration could be beneficial. Moreover, Figure \ref{fig:dependence_on_lambda_and_sigma} also indicates that the second condition in \eqref{eq:conditions_lambda_and_sigma}, which, in practice, requires $2\lambda - \sigma^2$ to be sufficiently large, may be redundant and could potentially be removed, as, e.g., in \cite{fornasier2022anisotropic} for minimization problems. However, this does not appear trivial in the present multiplayer game setting.

\begin{figure}[t]
	\centering
	\begin{subfigure}{0.9\textwidth}
		\includegraphics[width=1\linewidth]{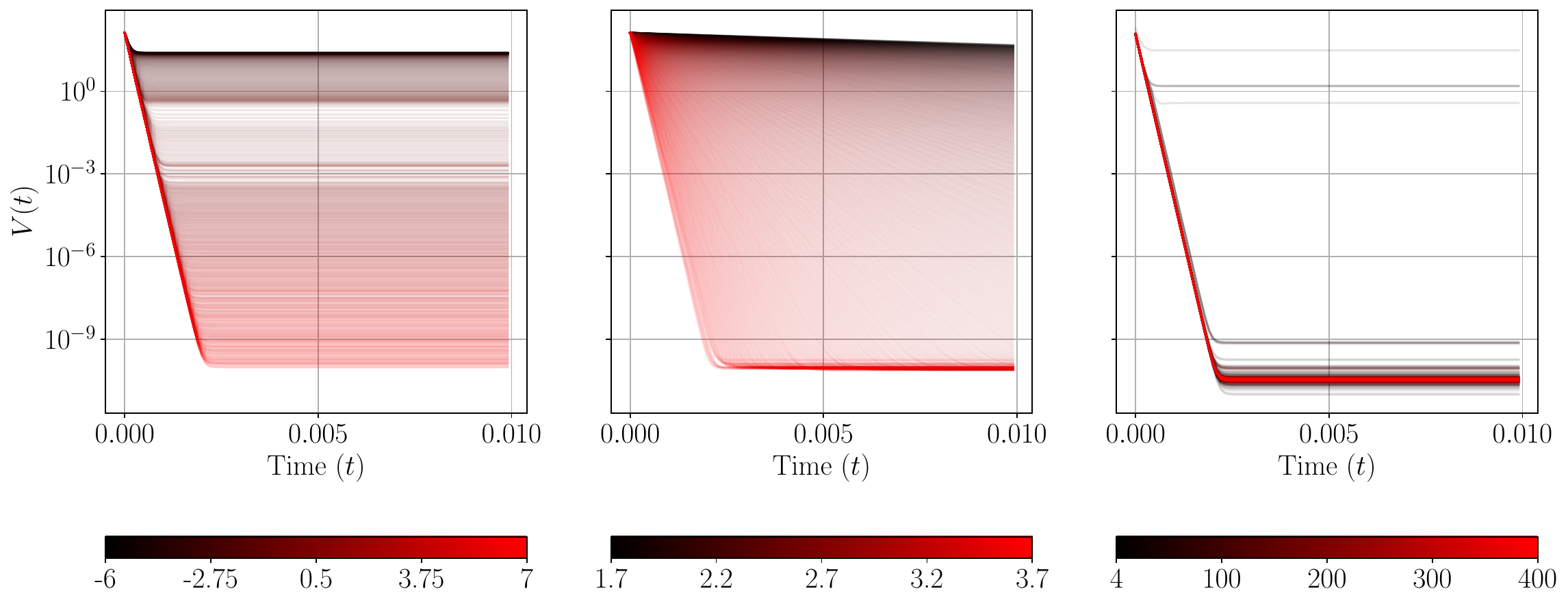}
		\caption{Analyzing variance reduction over time across different parameter choices. Left: Dependency on $\alpha$, see case \ref{item:dependence_on_alpha}, with each color associated to a logarithmic value of $\alpha$. Middle: Dependency on $\lambda$, see case \ref{item:dependence_on_lambda}. Once again, each color is associated to a logarithmic value of $\lambda$. Right: Dependency on $N$, see case \ref{item:dependence_on_N}.}
		\label{fig:dependency_on_aln}
	\end{subfigure}\\[0.1cm]
	\begin{subfigure}{0.9\textwidth}
		\includegraphics[width=1\linewidth]{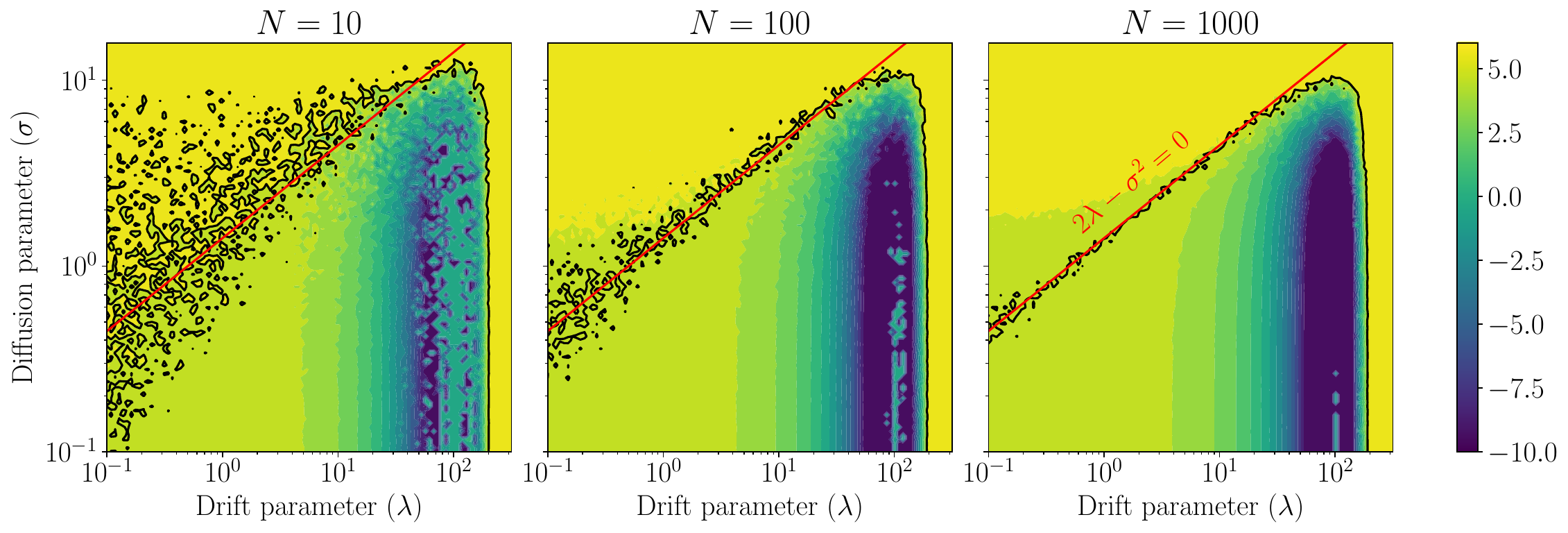}
		\caption{Logarithm of the variance at the final iteration for different parameter's choices. The black line corresponds to the value of $V(0)$. Confer case \ref{item:dependence_on_lambda_and_sigma} for details and Section \ref{sec:results_and_discussions} for comments.}
		\label{fig:dependence_on_lambda_and_sigma}
	\end{subfigure}
\caption{Studying the dependence of Algorithm \ref{alg:discretized_scheme} with respect to the algorithm's parameters to solve \eqref{eq:reward_i}.}
\end{figure}

\subsection{Nonlinear oligopoly games with several goods}

In this section, we present the numerical performance of Algorithm \ref{alg:discretized_scheme} to solve a nonlinear variant of a classical NE problem that arises from economics, especially, from Cournot's model of oligopoly games with several goods \cite{cournot1897researches}. We assume there are $M\in \N$ economic agents holding $d\in \N$ goods. For every $h \in [d]$, agent $i$ produces $x_{i, h}$ quantity of good $h$. Every product incurs into a production cost quantified by $c_h>0$, for all $h \in [d]$. In the present model, we assume that the quantity of each good influences the price of the others as well. In this setting, a typical price function is of the form
\begin{equation}
	p_h(\bx) =  \max\{a - b \ell_h^T(x_1+\dots + x_M), 0\}^\beta\,, \quad \text{for all} \ \bx:=(x_1,\dots, x_M)\in \R^{d\times M}\,,
\end{equation}
where $\beta, a, b>0$, and $\ell_h \in \R^d_+$ for all $h\in [d]$. Let us denote by $\boldsymbol{p}:\mathbb{R}^{d\times M}\to \mathbb{R}^d$ the function defined by $\boldsymbol{p}_h(\bx) = p_h(\bx)$ for all $h\in [d]$ and all $\bx:=(x_1,\dots, x_M)\in \R^{d\times M}$. The resulting cost function for agent $m$ is the difference between production cost and price, scaled with the quantity of good produced, namely
\begin{equation}\label{eq:reward_i}
	\mathcal{E}_m(x_m; \bx^{-m}) = \sum_{h=1}^d x_{m, h} (c_h - p_h(x_1, \dots, x_M)) =  x_m\cdot ( c - \boldsymbol{p}(\bx))\,, \quad \text{for all} \ m \in [M]\,,
\end{equation}
where $\cdot$ denotes the usual scalar product in $\R^d$. Note that the equilibrium problem with functions $\mathcal{E}_m$ defined in \eqref{eq:reward_i} is nonlinear and non-convex. In this section, we: (i) show that Algorithm \ref{alg:discretized_scheme} can indeed approximate an optimal solution efficiently; (ii) investigate the dependency of Algorithm \ref{alg:discretized_scheme} with respect to the dimension.

\subsubsection{Experimental setup}

To instantiate a synthetic problem, we proceed as follows. We set $\beta =2$, $a=10^2$, $b=10^{-3}$ and $L=3I_d + \mathbf{1}\mathbf{1}^*$, where $I_d$ is the identity in $\R^d$ and $\mathbf{1}:=(1,\dots, 1)\in \R^d$, and let $\ell_1,\dots, \ell_d$ be the rows of $L$. Note that these choices are somewhat arbitrary. Then, we set $\boldsymbol{\phi}\colon \R^d \to \R^d$ to be the function defined by
\begin{equation}
	\boldsymbol{\phi}_h(z) := \max \{a - b z_h, 0\}^\beta\,, \quad \text{for all} \ h \in [d], \ z=(z_1,\dots, z_d) \in \R^d\,,
\end{equation}
With this notation, the price function $\boldsymbol{p}$ can be written as $\boldsymbol{p}(\bx)=\boldsymbol{\phi}(L(x_1+\dots+x_M))$ for all $\bx=(x_1,\dots, x_M)\in \R^{d\times M}$, and we can write the optimality condition conveniently as
\begin{equation}\label{eq:optimality_conditions}
	0 = \partial_{x_m}\mathcal{E}_m(\bx) = c_m - \boldsymbol{\phi}\big(L(\bx_1+\dots+\bx_M)\big) - L^{*}\mathrm{D}\boldsymbol{\phi}\big(L(\bx_1+\dots+\bx_M)\big) \bx_m\,, \quad \text{for all} \ m \in [M]\,.
\end{equation}
We instantiate a synthetic problem by sampling the equilibrium point $\bx^*\in \R^{d\times M}$ uniformly in $[0, 10]^{d\times M}$ and computing $c$ via the right hand-side of \eqref{eq:optimality_conditions}, the above choice of $a, b$ and $L$ lead to $c\in \R^d_+$.

Once a NE problem is instantiated, we perform the following numerical experiments:

\begin{enumerate}
	\item (Anisotropic vs isotropic)\label{item:ani_vs_iso} First, we compare anisotropic against isotropic dynamics. To do so, we set $d=5$, $M=4$, $N=10^4$, $dt = 10^{-3}$, $\sigma = 1$, $\lambda = (10+\sigma^2)/2=5.5$, $\alpha = 10^{10}$, $K=1000$ and compare $20$ independent runs of Algorithm \ref{alg:discretized_scheme} with different initialization sampled from Gaussian distributions with variance $10$ and centered in $\bx^*+u$ where $u$ is sampled uniformly in $[-1, 1]^{d\times M}$ for each case. Then, we repeat the experiment with $\lambda = 50.5$.
	\item (Joint dependence on $N$ and $d$)\label{item:N_vs_d} We study the dependence of Algorithm \ref{alg:discretized_scheme} with respect to the dimension $d$ against the number of sample points $N$. In particular, we set $dt=10^{-4}$, $M=4$, $\sigma=10^{-1}$, $\lambda=(10^4+\sigma^2)/2$, $\alpha=10^{10}$, $K=15$ and let $d$ range from $2$ to $20$, and $N$ range from $2$ to $500$.
	\item (Joint dependence on $\alpha$ and $d$)\label{item:alpha_vs_d} We study the dependence of Algorithm \ref{alg:discretized_scheme} with respect to the dimension $d$ against $\alpha$. Specifically, we set $M=4$, $dt=10^{-4}$, $\sigma=10^{-1}$, $N=10^3$, $\lambda=(10^4+\sigma^2)/2$, $K=15$, and let $d$ range from $2$ to $20$ and consider $100$ choices of $\alpha$ spaced evenly in a logarithmic scale from $10$ to $10^{10}$.  
\end{enumerate}

We report in Figure \ref{fig:iso_vs_ani} the result of the experiment in case \ref{item:ani_vs_iso}, and in Figure \ref{fig:dependence_on_d} the results of cases \ref{item:N_vs_d} and \ref{item:alpha_vs_d}.

\subsubsection{Results and discussion}\label{sec:results_and_discussions_higher_dimensions}

In Figure \ref{fig:iso_vs_ani}, which shows the results corresponding to the experiment in case \ref{item:ani_vs_iso}, as observed for instance in \cite{fornasier2022convergence}, we see that the anisotropic case does indeed show a faster convergence rate especially in the initial iterations and for $\lambda$ which is of the order of $\sigma^2$. If $\lambda$ increases, though, we see no significant differences in the convergence behavior of anisotropic or isotropic dynamics, both for the residual and for the variance decrease.

Figure \ref{fig:dependence_on_d} shows the results of the comparisons in cases \ref{item:alpha_vs_d} and \ref{item:N_vs_d}. Regarding case \ref{item:N_vs_d}, we can see that indeed as the dimension increases, we need exponentially many particles (quantified by $N$) to reach high accuracy of order $10^{-9}$. This indicates that Algorithm \ref{alg:discretized_scheme}, while being quite reliable, and sometimes the only available solver for low-dimensional problems, might need further improvement in high-dimensional settings.

Regarding case \ref{item:alpha_vs_d}, Figure \ref{fig:dependence_on_d} (Right) shows that, while increasing $\alpha$ allows reaching a higher accuracy for fixed $d$ as shown in Figure \ref{fig:dependency_on_aln} (Left), the final accuracy also depends on $d$. Specifically, while for $d=2$ it suffices to set $\alpha=10^{3}$ to reach a final accuracy of order $10^{-10}$, for $d=5$ this drastically increases to $10^{5}$. For higher dimensions, further improvement might be needed  to guarantee convergence, though $N$ is set to $10^5$. 

\begin{figure}[t]
	\centering
	\begin{subfigure}{0.49\linewidth}
		\includegraphics[width=\linewidth]{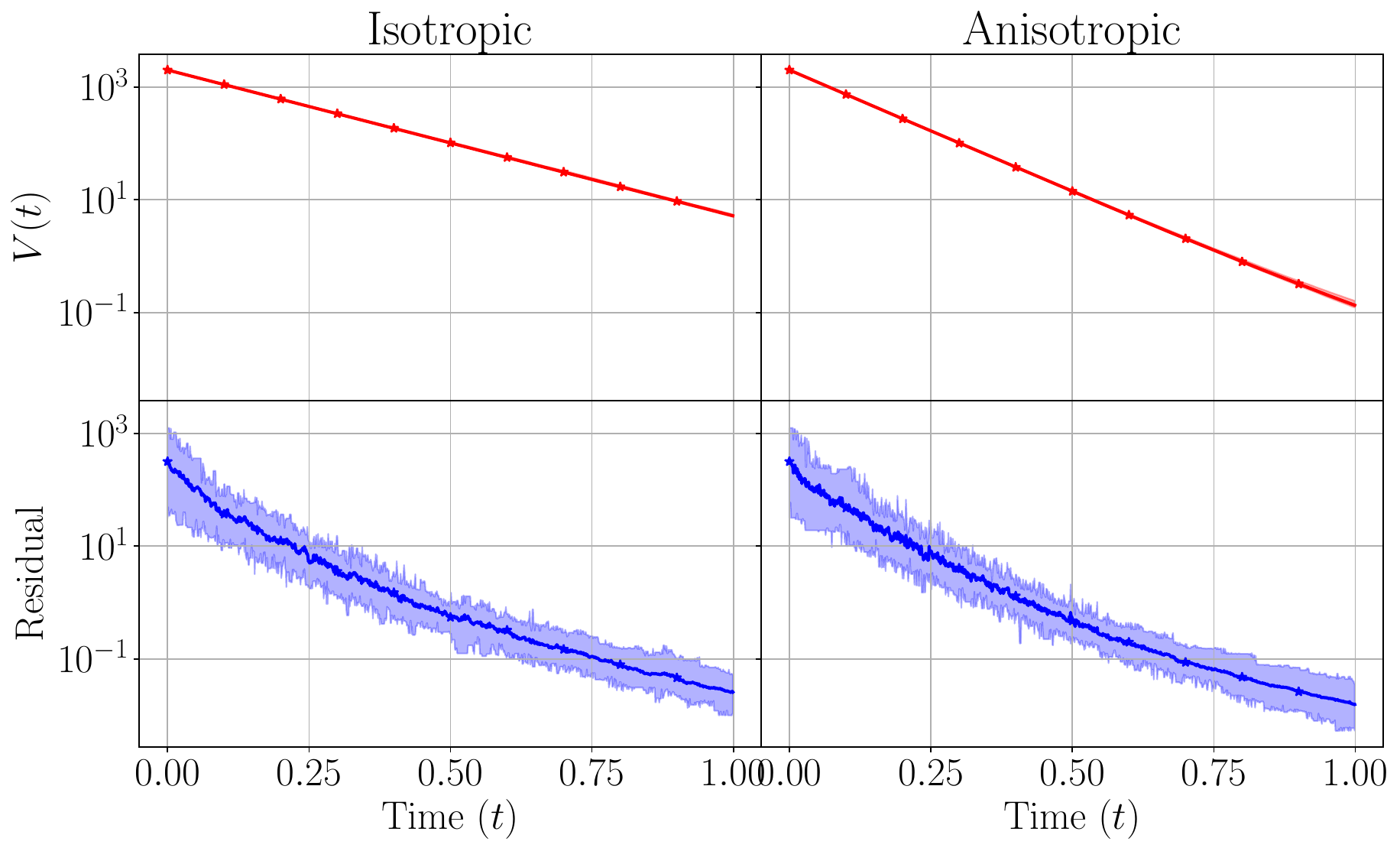}
		\caption{The case $\lambda= 5.5$.}
	\end{subfigure}
	\hfill
	\begin{subfigure}{0.49\linewidth}
		\includegraphics[width=\linewidth]{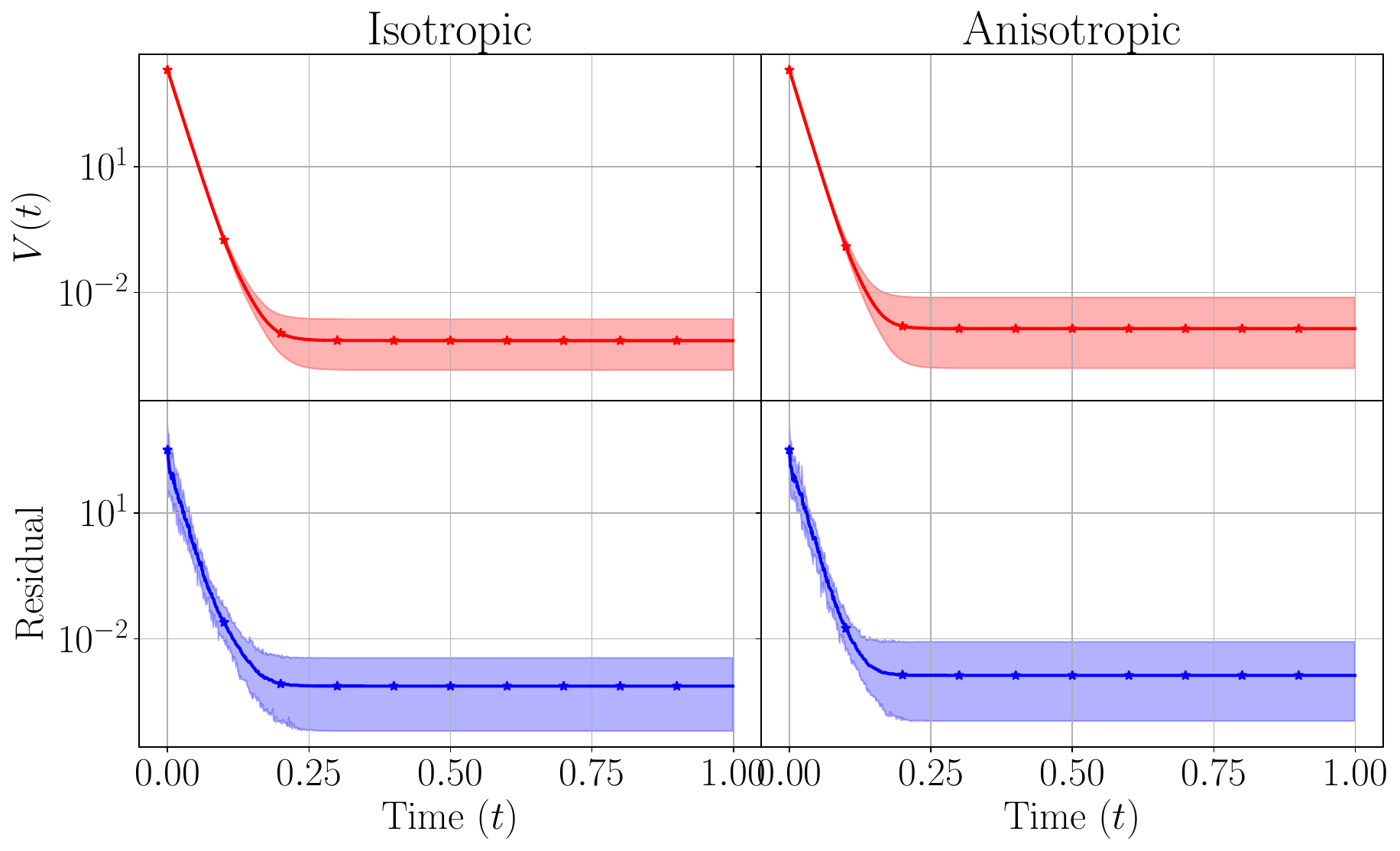}
		\caption{The case $\lambda=50.5$.}
	\end{subfigure}
	\caption{Comparing Variance and Residual (according to \eqref{eq:optimality_conditions}) as a function of time for isotropic and anisotropic dynamics in Algorithm \ref{alg:discretized_scheme} with two different choices of the drift parameter $\lambda$. Confer case \ref{item:ani_vs_iso} for details and to Section \ref{sec:results_and_discussions_higher_dimensions} for further comments.}
	\label{fig:iso_vs_ani}
\end{figure}

\begin{figure}[t]
\centering
	\begin{subfigure}{0.35\linewidth}
		\includegraphics[width=\linewidth]{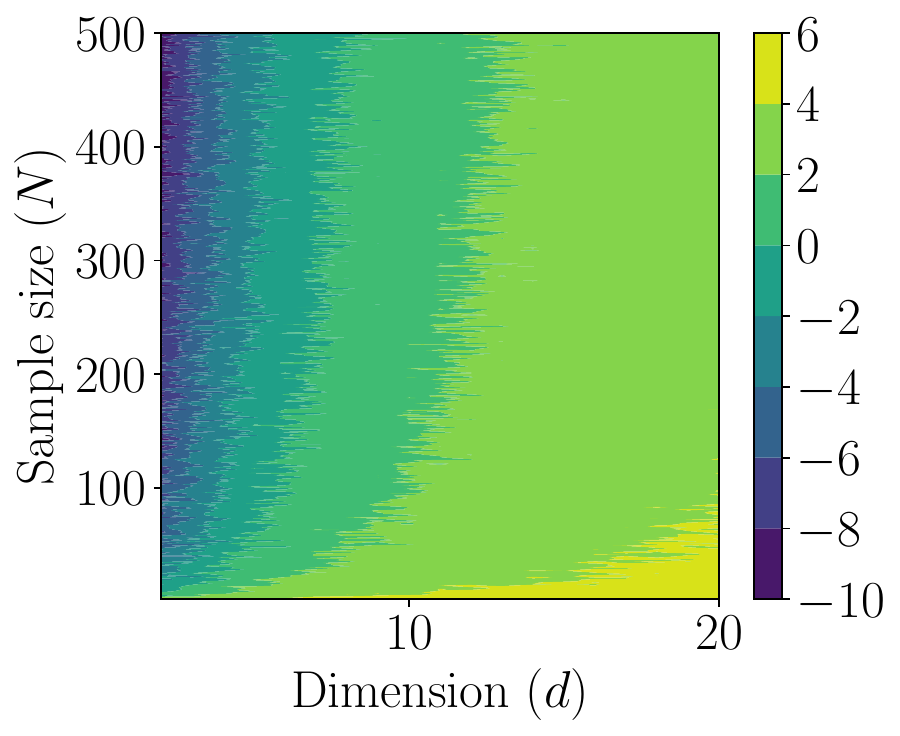}
	\end{subfigure}
	\begin{subfigure}{0.35\linewidth}
		\includegraphics[width=\linewidth]{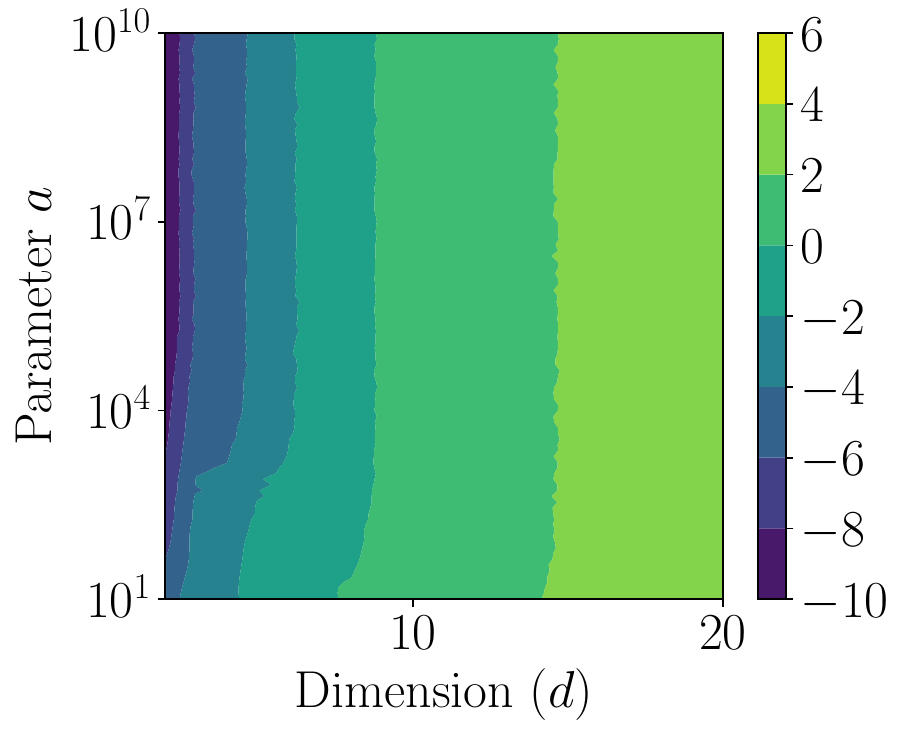}
	\end{subfigure}
	\caption{Logarithm of the variance at the final iteration for different parameter choices across different dimensions. Confer case \ref{item:alpha_vs_d} and \ref{item:N_vs_d} for details and to Section \ref{sec:results_and_discussions_higher_dimensions} for further comments.}
	\label{fig:dependence_on_d}
\end{figure}

\section{Conclusion}
In this paper, we present a novel variant of the consensus-based particle method aimed at identifying global NE within non-convex multiplayer games. We establish the global convergence of the dynamics to the global NE, denoted as $\bx^*=(x_1^*,\dots,x_m^*)$, of the cost functions $\{\TE_m\}_{m\in [M]}$ in mean-field law (large particle limit). Our analysis focuses on the time-evolution of the distance between the mean-field dynamics $\{X_t^m\}_{m\in [M]}$ and the unique global NE $\bx^*$, i.e. $V(t)=\sum_{m=1}^M \EE[|X_t^m-x_m^*|^2]$, demonstrating its exponential decay over time, achieving any prescribed accuracy $\varepsilon$ under suitable assumptions. We illustrate the numerical performance of this method using a one-dimensional example as well as real-world applications rooted in economics, particularly oligopoly games involving multiple goods. These examples confirm that the theoretical predictions are indeed reflected in practice, underscoring the method's competitive performance, particularly when dealing with highly nonlinear and non-convex problems in finite dimensions.

\section*{Acknowledgments} The Department of Mathematics and Scientific Computing at the University of Graz, with which E.C. and H.H. are affiliated, is a member of NAWI Graz (https://nawigraz.at/en). E.C. has received funding from the European Union’s Framework Programme for Research and Innovation Horizon 2020 (2014--2020) under the Marie Skłodowska--Curie Grant Agreement No. 861137. J.Q. is partially supported by the National Science and Engineering Research Council of Canada (NSERC) and by the 2023-2024 PIMS-Europe Fellowship.

\bibliography{multi} 
\bibliographystyle{ieeetr}

\section*{Appendix}

\begin{proof}[Proof of Lemma \ref{lemV}]
	By It\^o--Doeblin formula, we have
	\begin{equation*}
		\begin{aligned}
			d|X_t^m-x_m^*|^2&=2(X_t-x_m^*)\cdot dX_t+\sigma^2|X_t^m-X_\alpha(\rho_t^m)|^2dt\\
			&=-2\lambda(X_t-x_m^*)\cdot (X_t-X_\alpha(\rho^m_t))dt\\
			& \qquad +\sigma(X_t-x_m^*)\cdot D( X_t-X_\alpha(\rho^m_t)) dB_t^m + \sigma^2|X_t^m-X_\alpha(\rho_t^m)|^2dt\,,
		\end{aligned}
	\end{equation*}
	which implies
	\begin{align}
		\frac{d  V^m(t)}{dt}&=-2\lambda\EE[(X_t^m-x_m^*)\cdot(X_t^m-X_\alpha(\rho_t^m))]+\sigma^2\EE[|X_t^m-X_\alpha(\rho_t^m)|^2]\nn\\
		&=-2\lambda V^m(t)-2\lambda\EE[(X_t^m-x_m^*)\cdot(x_m^*-X_\alpha(\rho_t^m))]+\sigma^2\EE[|X_t^m-X_\alpha(\rho_t^m)|^2]\nn\\
		&=-(2\lambda-\sigma^2) V^m(t)+2(\sigma^2-\lambda)\EE[(X_t^m-x_m^*)\cdot(x_m^*-X_\alpha(\rho_t^m))]+\sigma^2|x_m^*-X_\alpha(\rho_t^m)|^2\nn\\
		&\leq -(2\lambda-\sigma^2) V^m(t)+2(\lambda+\sigma^2)V^m(t)^{\frac{1}{2}}|x_m^*-X_\alpha(\rho_t^m)|+\sigma^2 |x_m^*-X_\alpha(\rho_t^m)|^2\,,
	\end{align}
	where, in the first identity, the stochastic integral  $\int \sigma(X_t-x_m^*)\cdot D( X_t-X_\alpha(\rho^m_t)) dB_t^m $ can be easily seen to vanish since $\EE[|X_t^m|^4]<\infty$.
\end{proof}

\begin{proof}[Proof of Proposition \ref{propositive}]
With direct calculations, see, e.g., \cite{fornasier2021consensus1}, we have $\mbox{Image}(\phi_r^{\bar\bx}) = [0, 1]$, $\mbox{supp}(\phi_r) = \boldsymbol{B}_r(\bar \bx)$, $\phi_r^{\bar\bx}\in \mc{C}_c^\infty(\RR^{Md})$ and for all $\bx:=(x_1,\dots x_M) \in \R^{Md}$
\begin{equation}\label{eq:derivatives_of_mollifier}
	\begin{aligned}
		\partial_{(x_m)_k} \phi_{r}^{\bar\bx}(\bx) &=-2 r^{2} \frac{(x_m-\bar{x}_m)_{k}}{\big(r^{2}-(x_m-\bar{x}_m)_{k}^{2}\big)^{2}} \phi_{r}^{\bar\bx}(\bx)\,,\\
		\partial_{(x_m)_k^2}^{2} \phi_{r}^{\bar\bx}(\bx) &=2 r^{2}\bigg(\frac{2\big(2(x_m-\bar{x}_m)_{k}^{2}-r^{2}\big)(x_m-\bar{x}_m)_{k}^{2}-\big(r^{2}-(x_m-\bar{x}_m)_{k}^{2}\big)^{2}}{\big(r^{2}-(x_m-\bar{x}_m)_{k}^{2}\big)^{4}}\bigg) \phi_{r}^{\bar\bx}(\bx)\,.
	\end{aligned}
\end{equation}
	It is easy to compute 
\begin{equation}\label{eq:first_estimate}
	\frac{d \EE[\phi_r^{\bar\bx}(\bX_t)]}{dt}=\sum_{m=1}^{M}\sum_{k=1}^{d}\EE[T_{1,k}^m+T_{2,k}^m]\,,
\end{equation}
where $T_{1,k}^m$ and $T_{2, k}^m$ are random variables defined for all $k \in [d]$ and $m \in [M]$ as
\begin{equation*}
	T_{1,k}^m:=-\lambda(X_t^m-X_\alpha(\rho_t^m))_{k}\partial_{(x_m)_k}\phi_r^{\bar\bx}(\bX_t)\,, \quad \text{and} \quad
	T_{2,k}^m:=\frac{\sigma^2}{2}(X_t^m-X_\alpha(\rho_t^m))_{k}^2\partial_{(x_m)_k^2}^2\phi_r^{\bar\bx}(\bX_t)\,.
\end{equation*}
Let us now define for each $k\in [d]$ and $m\in [M]$ the subsets
\begin{equation}
	\begin{aligned}
		K_{1,k}^m&:=\{\bx\in \R^{Md}:\,|(x_m-\bar{x}_m)_k|>\sqrt{c}r\}\,,\\
		K_{2,k}^m&:=\bigg\{\bx \in \mathbb{R}^{Md}:-\lambda(x_m-X_{\alpha}(\rho_{t}^m)_{k}(x_m-\bar{x}_m)_{k}(r^{2}-(x_m-\bar{x}_m)_{k}^{2})^{2} \\
		& \hspace{3cm} >\tilde{c} r^{2} \frac{\sigma^{2}}{2}(x_m-X_{\alpha}(\rho_{t}^m))_{k}^{2}(x_m-\bar{x}_m)_{k}^{2}
		\bigg\}\,,
	\end{aligned}
\end{equation}
where $\tilde c:=2c-1 $ with $(\sqrt{5}-1)/2\leq c<1$. For fixed $k$ and $m$ we decompose $\boldsymbol{B}_r(\bx^*)$ according to
\begin{equation}
	\boldsymbol{B}_r(\bx^*)=\underbrace{\left((K_{1,k}^m)^{c} \cap \boldsymbol{B}_r(\bx^*)\right)}_{=: \Omega_1} \cup\underbrace{\left(K_{1,k}^m \cap (K_{2,k}^m)^{c} \cap \boldsymbol{B}_r(\bx^*)\right)}_{=: \Omega_2} \cup\underbrace{\left(K_{1,k}^m \cap K_{2,k}^m \cap \boldsymbol{B}_r(\bx^*)\right)}_{=: \Omega_3}\,.
\end{equation}
Then, in the following we consider each of the cases $\bX_t\in\Omega_1$, $\bX_t\in\Omega_2$ and $\bX_t\in \Omega_3$ separately. In the following, to simplify notation, we denote by $\phi_r^{\bar\bx}=\phi_r^{\bar\bx}(\bX_t)$ whenever this does not create confusion.

\textbf{Subset} $\Omega_1:$ For each $\bX_t\in (K_{1,k}^m)^c$, we have that $|(X_t^m-\bar x_m)|\leq \sqrt{c}r$ holds. To estimate $T_{1,k}^m$, we use the expression of $\partial_{x_k}\phi_r^{\bar\bx}$ in \eqref{eq:derivatives_of_mollifier} and get
\begin{align*}
	T_{1,k}^m &=2 r^{2} \lambda(X_t^m-X_{\alpha}\left(\rho_{t}^m\right))_{k} \frac{(X_t^m-\bar{x}_m)_{k}}{\big(r^{2}-(X_t^m-\bar{x}_m)_{k}^{2}\big)^{2}} \phi_{r}^{\bar\bx} \\
	& \geq-2 r^{2} \lambda \frac{ |(X_t^m-X_{\alpha}(\rho_{t}^m))_{k}| |(X_t^m-\bar{x}_m)_{k}|}{(r^{2}-(X_t^m-\bar{x}_m)_{k}^{2})^{2}} \phi_{r}^{\bar\bx} \geq-\frac{2 \lambda(\sqrt{c} r+B) \sqrt{c}}{(1-c)^{2} r} \phi_{r}^{\bar\bx} =:-q_{1}^m \phi_{r}^{\bar\bx}\,,
\end{align*}
where in the last inequality we have used 
\begin{align*}
	&|(X_t^m-X_{\alpha}\left(\rho_{t}^m\right))_{k}| \leq |(X_t^m-\bar{x}_m)_{k}|+ |\left(\bar{x}_m-X_{\alpha}\left(\rho_{t}^m\right)\right)_{k} | \\
	\leq & \sqrt{c} r+   |\left(\bar{x}_m-x_m^*\right)_k|+|\left(x_m^*-X_{\alpha}\left(\rho_{t}^m\right)\right)_{k} |\leq \sqrt{c} r+\barc_2+|x_m^*| +B\\
	=:&\sqrt{c} r+\mc B\,.
\end{align*}
Similarly for $T_{2,k}^m$ one obtains
\begin{align*}
	T_{2,k}^m &=\sigma^{2} r^{2}(X_t^m-X_{\alpha}(\rho_{t}^m))_{k}^{2} \frac{2 (2 (X_t^m-\bar{x}_m)_{k}^{2}-r^{2})(X_t^m-\bar{x}_m)_{k}^{2}-(r^{2}-(X_t^m-\bar{x}_m)_{k}^{2})^{2}}{(r^{2}-(X_t^m-\bar{x}_m)_{k}^{2})^{4}} \phi_{r}^{\bar\bx} \\
	& \geq-\frac{2 \sigma^{2}(c r^{2}+\mc B^{2})(2c^2-c+1)}{(1-c)^{4} r^{2}} \phi_{r}^{\bar\bx}=:-q_{2}^m \phi_{r}^{\bar\bx}\,.
\end{align*}

\textbf{Subset} $\Omega_2:$  As  $\bX_t\in K_{1,k}^m$ we have $|(X_t^m-\bar{x}_m)_k|>\sqrt{c}r$. We observe that $T_{1,k}^m+T_{2,k}^m\geq 0$ for all $X_t^m$ satisfying
\begin{align}\label{positve}
	&\bigg(-\lambda (X_t^m-X_{\alpha} (\rho_{t}^m ) )_{k} (X_t^m-\bar{x}_m )_{k}+\frac{\sigma^{2}}{2} (X_t^m-X_{\alpha} (\rho_{t}^m ) )_{k}^{2}\bigg) (r^{2}- (X_t^m-\bar{x}_m )_{k}^{2} )^{2} \nn\\
	&\hspace{1cm}	\leq \sigma^{2} (X_t^m-X_{\alpha} (\rho_{t}^m ) )_{k}^{2} (2 (X_t^m-\bar{x}_m )_{k}^{2}-r^{2} ) (X_t^m-\bar{x}_m )_{k}^{2}\,.
\end{align}
Actually, this can be verified by first showing that
{\small \begin{align*}
	&-\lambda (X_t^m-X_{\alpha} (\rho_{t}^m ) )_{k} (X_t^m-\bar{x}_m )_{k} (r^{2}- (X_t^m-\bar{x}_m )_{k}^{2} )^{2}\leq \tilde{c} r^{2} \frac{\sigma^{2}}{2} (X_t^m-X_{\alpha} (\rho_{t}^m ) )_{k}^{2} (X_t^m-\bar{x}_m )_{k}^{2}\\
	&=(2 c-1) r^{2} \frac{\sigma^{2}}{2} (X_t^m-X_{\alpha} (\rho_{t}^m ) )_{k}^{2} (X_t^m-\bar{x}_m )_{k}^{2} \leq (2 (X_t^m-\bar{x}_m )_{k}^{2}-r^{2} ) \frac{\sigma^{2}}{2} (X_t^m-X_{\alpha} (\rho_{t}^m ) )_{k}^{2} (X_t^m-\bar{x}_m )_{k}^{2}\,,
\end{align*}}
where we have used the fact that $\bX_t\in K_{1,k}^m\cap( K_{2,k}^m)^c$ and $\tilde c=2c-1$. One may also notice that
\begin{align*}
	\frac{\sigma^{2}}{2} & (X_t^m-X_{\alpha} (\rho_{t}^m ) )_{k}^{2} (r^{2}- (X_t^m-\bar{x}_m )_{k}^{2} )^{2} \leq \frac{\sigma^{2}}{2} (X_t^m-X_{\alpha} (\rho_{t}^m ) )_{k}^{2}(1-c)^{2} r^{4} \\
	& \leq \frac{\sigma^{2}}{2} (X_t^m-X_{\alpha} (\rho_{t}^m ) )_{k}^{2}(2 c-1) r^{2} c r^{2}\\
	&\leq \frac{\sigma^{2}}{2} (X_t^m-X_{\alpha} (\rho_{t}^m ) )_{k}^{2} (2 (X_t^m-\bar{x}_m )_{k}^{2}-r^{2} ) (X_t^m-\bar{x}_m )_{k}^{2}\,,
\end{align*}
by using $(1-c)^2\leq(2c-1)c$, namely $\frac{\sqrt{5}-1}{2}\leq c<1$. Hence \eqref{positve} holds and we have $T_{1,k}^m+T_{2,k}^m\geq 0$.

\textbf{Subset} $\Omega_3:$ Notice that when $(X_t^m)_k=(X_\alpha(\rho_t^m))_k$, we have $T_{1,k}^m=T_{2,k}^m=0$, so in this case there is nothing to prove. If $(X_t^m)_k\neq(X_\alpha(\rho_t^m))_k$, or $\sigma^2(X_t-X_\alpha(\rho_t^m))_k^2>0$ $(\sigma>0)$, we exploit
$\bX_t\in K_{2,k}^m$ to get
\begin{align*}
	\frac{ (X_t^m-X_{\alpha} (\rho_{t}^m ) )_{k} (X_t^m-\bar{x}_m )_{k}}{ (r^{2}- (X_t^m-\bar{x}_m )_{k}^{2} )^{2}} & \geq \frac{- | (X_t^m-X_{\alpha} (\rho_{t}^m ) )_{k} | | (X_t^m-\bar{x}_m )_{k} |}{ (r^{2}- (X_t^m-\bar{x}_m )_{k}^{2} )^{2}} \\
	&>\frac{2 \lambda (X_t^m-X_{\alpha} (\rho_{t}^m ) )_{k} (X_t^m-\bar{x}_m )_{k}}{\tilde{c} r^{2} \sigma^{2} | (X_t^m-X_{\alpha} (\rho_{t}^m ) )_{k} | | (X_t^m-\bar{x}_m )_{k} |} \geq-\frac{2 \lambda}{\tilde{c} r^{2} \sigma^{2}} \,.
\end{align*}
Using this, $T_{1,k}^m$ can be bounded from below
\begin{equation}
	T_{1,k}^m=2 r^{2} \lambda (X_t^m-X_{\alpha} (\rho_{t}^m ) )_{k} \frac{ (X_t^m-\bar{x}_m )_{k}}{ (r^{2}- (X_t^m-\bar{x}_m )_{k}^{2} )^{2}} \phi_{r}^{\bar\bx}\geq-\frac{4 \lambda^{2}}{\tilde{c} \sigma^{2}} \phi_{r}^{\bar\bx}=:-q_{3}^m \phi_{r}^{\bar\bx}\,.
\end{equation}
Moreover since $\bX_t\in K_{1,k}^m$ and $2(2c-1)c\geq(1-c)^2$ implied by the assumption, one has
\begin{equation}
	2 (2 (X_t^m-\bar{x}_m )_{k}^{2}-r^{2} ) (X_t^m-\bar{x}_m )_{k}^{2} \geq (r^{2}- (X_t^m-\bar{x}_m )_{k}^{2} )^{2}\,,
\end{equation}
which yields that $T_{2,k}^m\geq 0$. 

\textbf{Concluding the proof:} Collecting estimates from above, from \eqref{eq:first_estimate} we get
\begin{equation*}
	\begin{aligned}
		\frac{d \EE[\phi_r^{\bar\bx}(\bX_t)]}{dt}&=\sum_{m=1}^{M}\sum_{k=1}^{d}\EE[T_{1,k}^m+T_{2,k}^m]\nn\\
		&=\sum_{m=1}^{M}\sum_{k=1}^{d}\bigg(\EE[(T_{1,k}^m+T_{2,k}^m)\textbf{I}_{\Omega_1}]+\EE[(T_{1,k}^m+T_{2,k}^m)\textbf{I}_{\Omega_2}]+\EE[(T_{1,k}^m+T_{2,k}^m)\textbf{I}_{\Omega_3}]\bigg)\nn\\
		&\geq-d\sum_{m=1}^M (q_1^m+q_2^m+q_3^m)\EE[\phi_r^{\bar\bx}(\bX_t)]=-\vartheta \EE[\phi_r^{\bar\bx}(\bX_t)]\,,
	\end{aligned}
\end{equation*}
and an application of Gronwall's inequality and taking infimum on $\bar\bx$ yield \eqref{eq:propositive}.
\end{proof}

\end{document}